\documentclass[11pt]{article}
\usepackage{amsfonts}
\usepackage{amssymb}
\usepackage{amsthm}
\usepackage{amsmath}
\usepackage{graphicx}
\usepackage{empheq}
\usepackage{indentfirst}
\usepackage{cite}
\usepackage{mathrsfs}
\usepackage{graphics}
\usepackage{xcolor}
\newtheorem{theorem}{\textbf{Theorem}}[section]
\newtheorem{lemma}{\textbf{Lemma}}[section]
\newtheorem{proposition}{\textbf{Proposition}}[section]
\newtheorem{corollary}{\textbf{Corollary}}[section]
\newtheorem{remark}{\textbf{Remark}}[section]
\newtheorem{definition}{\textbf{Definition}}[section]

\allowdisplaybreaks[4]

\def\be{\begin{equation}}
\def\ee{\end{equation}}
\def\bea{\begin{eqnarray}}
\def\eea{\end{eqnarray}}
\def\bt{\begin{theorem}}
\def\et{\end{theorem}}
\def\bl{\begin{lemma}}
\def\el{\end{lemma}}
\def\br{\begin{remark}}
\def\er{\end{remark}}
\def\bp{\begin{proposition}}
\def\ep{\end{proposition}}
\def\bc{\begin{corollary}}
\def\ec{\end{corollary}}
\def\bd{\begin{definition}}
\def\ed{\end{definition}}

\def\non{\nonumber }

\begin{document}

\title{Robust exponential attractors \\ for the modified phase-field crystal equation}

\author{
{\sc Maurizio Grasselli}\footnote{Dipartimento di Matematica, Politecnico di Milano, Milano 20133, Italy, \textit{maurizio.grasselli@polimi.it}}
\ and
\ {\sc Hao Wu}\footnote{School of Mathematical Sciences and Shanghai Key Laboratory for Contemporary Applied Mathematics, Fudan University, Shanghai 200433, China, \textit{haowufd@yahoo.com}}
}

\date{\today}

\maketitle


\begin{abstract}
\noindent We consider the modified phase-field crystal (MPFC) equation that has recently been proposed by P. Stefanovic et al. This is a variant of the phase-field crystal (PFC) equation, introduced by K.-R. Elder et al., which is characterized by the presence of an inertial term $\beta\phi_{tt}$. Here $\phi$ is the phase function standing for the number density of atoms and $\beta\geq 0$ is a relaxation time. The associated dynamical system for the MPFC equation with respect to the parameter $\beta$ is analyzed. More precisely, we establish the existence of a family of exponential
attractors $\mathcal{M}_\beta$ that are H\"older continuous with respect to $\beta$.


\end{abstract}

\section{Introduction}
\setcounter{equation}{0}
\noindent
In this paper, we consider the following modified phase-field crystal (MPFC) equation (see e.g., \cite{SHP06,GGKE,WW10,WW11}):
 \be
\beta \phi_{tt}+\phi_t=\Delta[\Delta^2 \phi+2\Delta \phi+f(\phi)],\quad \text{ in } \; Q\times (0,+\infty),\label{e1}
 \ee
where the nonlinearity $f$ is given by
 \be
 f(\phi)=\phi^3+(1-\epsilon)\phi\label{f}
 \ee
 and the domain $Q$ is supposed to be $Q=(0,1)^n \subset \mathbb{R}^n$, $n\leq 3$. The phase function $\phi$ approximates the number density of atoms in the material occupying $Q$. The positive parameter $\beta$ is the relaxation time and $\epsilon$ is a positive constant such that $\epsilon \sim T_e-T$, with $T_e$ being the equilibrium temperature at which the phase transition occurs.
 Equation \eqref{e1} is considered in the periodic setting for the sake of simplicity and it is subject to the initial conditions
\be
\phi|_{t=0}=\phi_0(x),\quad \phi_t|_{t=0}=\phi_1(x), \qquad x\in Q.\label{e2}
\ee

Recently, the so-called phase-field crystal (PFC) approach has been employed to model and simulate the dynamics of crystalline materials, including crystal growth in a supercooled liquid, dendritic and eutectic solidification, epitaxial growth, and so on (see, e.g., \cite{BRV, EG,EKHG,EGL11,ELWGTTG12,OS08,PDA}).  The corresponding PFC equation takes the following form
\be
\phi_t=\Delta[\Delta^2 \phi+2\Delta \phi+f(\phi)], \quad \text{ in } \; Q\times (0,+\infty).\label{pfc}
\ee
As before we endow the equation with periodic boundary conditions and we assume an initial condition for $\phi$:
\be
\phi|_{t=0}=\phi_0(x), \quad x\in Q.\label{pfc2}
\ee
The PFC equation \eqref{pfc} can be viewed as a (conserved) gradient flow generated by the Fr\'{e}chet derivative of the following (dimensionless) free energy functional (cf. \cite{EKHG, EG}), that is,
\be
 E(\phi)=\int_{Q} \left(\frac12|\Delta \phi|^2-|\nabla \phi|^2+F(\phi)\right) dx,\label{EE}
\ee
with
 \be
F(\phi)=\frac{1-\epsilon}{2}\phi^2+\frac14 \phi^4.\label{F}
 \ee
 \par Equation \eqref{pfc} describes the microstructure of solid-liquid systems at inter-atomic length scales and provides a possibly accurate way to model crystal dynamics, especially defect dynamics in atomic-scale resolution (see, for instance, \cite{ELWGTTG12} and the references therein).  However, it fails to distinguish between the elastic relaxation and diffusion time scales (cf., e.g., \cite{EG, SHP06}). The MPFC equation \eqref{e1} was proposed in \cite{SHP06} in order to overcome this difficulty and to incorporate both the faster elastic relaxation (e.g., in a rapid quasi-phononic time scale) and the slower mass diffusion (see also \cite{GDL09,GE11,SHP09} and \cite[Section 3.1.1.2]{ELWGTTG12}). The MPFC equation \eqref{e1} can be viewed as a singular perturbation of the PFC equation \eqref{pfc}. We recall that (singular) perturbations of this kind have been derived in the phase-field literature  to take large deviations from thermodynamic equilibrium into account (cf., e.g., \cite{CKLE,GJ05}). The presence of the inertial term $\beta \phi_{tt}$ is a nontrivial modification from the mathematical point of view. Indeed, contrary to the sixth-order parabolic type equation \eqref{pfc}, solutions to \eqref{e1} do not regularize in finite time. Besides,  the MPFC equation \eqref{e1} does not
have a gradient structure in general, i.e., there is no Lyapunov functional for \eqref{e1} for general initial data, which is decreasing in time (see \cite{GW3} for further technical details).

  The MPFC equation \eqref{e1} has been studied numerically in \cite{BHLWWZ,BLWW,GGKE,WW11} (see also, e.g., \cite{BRV,CW,GN,WWL09,HWWL09} for the PFC equation \eqref{pfc}). In particular, the authors have derived different types of unconditionally energy stable finite difference schemes based on a suitable convex splitting for the free energy $E$. Concerning the theoretical study of the MPFC equation \eqref{e1}, existence of a \emph{weak} solution and of a unique \emph{strong} solution up to any positive final time $T>0$ were proven in \cite{WW10} by using a time discretization scheme and taking the special initial value of $\phi_t$ (equal to zero) in order to ensure the mass conservation property. Recently, existence and uniqueness of \emph{energy} solutions to the MPFC equation \eqref{e1} without any restriction on the initial value of $\phi_t$ has been established in \cite{GW3} (see Theorem \ref{exe} below). We recall that such solutions are natural and more general than the weak ones. Moreover, in \cite{GW3}, the authors have also proven the existence of global and exponential attractors as well as the convergence of single trajectories to single stationary states for any fixed parameter $\beta>0$. In this case, it is worth noting that the mass is conserved only asymptotically and the corresponding (dissipative) dynamical system is no longer a gradient-like system (i.e., there is no Lyapunov functional).

An interesting open issue is the construction of a robust family of exponential attractors with respect to the relaxation time $\beta$ (see \cite[Section 3.3]{MZ} for this concept and references therein). This is precisely the goal of the present paper, that is, the existence of
a family of exponential attractors depending on $\beta\geq 0$, which is (H\"{o}lder) continuous with respect to $\beta$.
Such a result essentially says that the non-transient dynamics of the MPFC equation \eqref{e1} is \emph{close} to the one
of the PFC equation \eqref{pfc} in a quantitative way. The proof of this result is based on a general argument devised in
\cite{MPZ} for the damped semilinear wave equation, in which the singular perturbation can be treated by using a simple scaling argument. However, in the present case, the problem is much more complicated because of the additional difficulty related to the high (spatial) regularity gap between the phase function $\phi$ and its time derivative $\phi_t$. More precisely, the crucial estimate on energy norms of the difference between the solutions to the MPCF and the PFC equations, respectively, requires the construction of a sufficiently smooth invariant attracting set for the corresponding dynamical system (see Proposition \ref{abs2}).

The plan of this paper goes as follows. In the next section, after some preliminaries, we state the main result of this paper (see Theorem \ref{main}). Section~\ref{estimates}
contains the basic uniform \emph{a priori} estimates. The dissipative dynamical systems associated with MPFC and PFC equations are analyzed in Section~\ref{dissdyn}. Finally, the main theorem is proven in Section~\ref{robust}.


\section{Preliminaries and main result}
\setcounter{equation}{0}
\subsection{Notations and functional spaces}
We denote by $H^m_p(Q)$, $m\in \mathbb{N}$, the space of $H^m_{loc}(\mathbb{R}^n)$ functions that are $Q$-periodic. For an arbitrary $m\in \mathbb{N}$, $H^m_p(Q)$ is a Hilbert space with respect to the scalar product $(u,v)_{m}=\sum_{|\kappa|\leq m}\int_Q D^\kappa u(x)D^\kappa v(x) dx$ ($\kappa$ being a multi-index) and the induced norm $\|u\|_m=\sqrt{(u,u)_m}$.  For $m=0$, $H^0_p(Q)=L^2_p(Q)$ and the inner product as well as the norm on $L^2_p(Q)$ are simply indicated by $(\cdot, \cdot)$ and $\|\cdot\|$, respectively.

The mean value of any function $u\in L^2_p(Q)$ is denoted by $\langle u\rangle=|Q|^{-1}\int_{Q} u dx$ and we set $\overline u=u-\langle u\rangle$.
The dual space of $H^{m}_p(Q)$ is denoted by $H^{-m}_p(Q)$, which is equipped with the operator norm given by
 $\|\mathcal{T}\|_{-m}=\sup_{\|u\|_m=1,\ u\in H^m_p(Q)}|\mathcal{T}(u)|$.
 For $m=1$, we introduce an equivalent and more convenient norm associated with the inner product
$$(u, v)_{-1}= (\nabla \psi_u, \nabla \psi_v)+\langle u\rangle\langle v\rangle,\quad \forall\, u, v\in H^{-1}_p(Q),$$
where $\psi_u$ (respectively $\psi_v$) is the unique  solution with zero mean to the elliptic equation in $Q$ subject to periodic boundary conditions:
$$ -\Delta \psi_u=u-\langle u\rangle.$$
We also set $\dot{H}^m_p(Q)=\{u\in H^m_p(Q):\ \langle u\rangle=0\}$ and, for any $u, v\in \dot{L}^2_p(Q)$, we define
$$(u, v)_{-1}=(\nabla \psi_u, \nabla \psi_v)\quad \text{and}\quad \|u\|_{-1}=\|\nabla \psi_u\|.$$
Then we observe that $A_0=-\Delta: \dot{H}_p^2(Q)\mapsto \dot{L}^2_p(Q)$ is a positive operator and its powers $A_0^s$ ($s\in\mathbb{R}$) are well defined. In particular, for $s=-1$,
$$ (u, v)_{-1}=(A_0^{-1} u, v)=(u, A^{-1}_0 v)=(A_0^{-\frac12} u, A_0^{-\frac12} v) \quad \text{and}\quad \|u\|_{-1}=\|A_0^{-\frac12} u\|.$$
Finally, for any given $\beta>0$,  we introduce the product spaces
\bea
&& \mathbb{X}_0^\beta=H^2_p(Q)\times \sqrt{\beta}H^{-1}_p(Q),\quad \mathbb{X}_1^\beta=H^3_p(Q)\times \sqrt{\beta} L^2_p(Q),\non\\
&& \mathbb{X}_2^\beta=H^4_p(Q)\times \sqrt{\beta} H^1_p(Q),\quad\ \ \mathbb{X}_3^\beta=H^5_p(Q)\times \sqrt{\beta} H^2_p(Q).\non
\eea
These spaces are complete with respect to the metrics induced by the following norms, respectively,
 \bea
 &&\|(u,v)\|_{\mathbb{X}_0^\beta}=(\|u\|_{2}^2+\beta\|v\|_{-1}^2)^\frac12, \quad \|(u,v)\|_{\mathbb{X}_1^\beta}=(\|u\|_{3}^2+\beta\|v\|^2)^\frac12,\non\\
 && \|(u,v)\|_{\mathbb{X}_2^\beta}=(\|u\|_{4}^2+\beta\|v\|_{1}^2)^\frac12,\quad \ \ \|(u,v)\|_{\mathbb{X}_3^\beta}=(\|u\|_{5}^2+\beta\|v\|_{2}^2)^\frac12.\non
 \eea
We note that for $\beta=0$, the second component of $\mathbb{X}_i^\beta$ ($i=0,1,2,3$) is simply set to be the null space $\{0\}$.

We also recall the definitions of the \textit{Hausdorff semidistance} as well as the \textit{symmetric Hausdorff distance} of two subset $U_1, U_2$ of a Banach space $V$ with metric $\mathrm{d}$, namely,
\bea
&&{\rm dist}_{V}(U_1,U_2)=\sup_{u_1\in U_1}\inf_{u_2\in U_2} \mathrm{d}(u_1, u_2),\non\\
&&{\rm dist}_{V}^{{\rm sym}}(U_1,U_2)=\max\{{\rm dist}_{V}(U_1,U_2), {\rm dist}_{V}(U_2,U_1)\}.\non
\eea

In the remaining part of the paper, if it is not otherwise stated, we indicate by $C$ or $C_i$, $i\in \mathbb{N}$, generic positive
constants depending only on structural quantities, independent of the parameter $\beta$. The symbol $c_Q$ will denote some embedding
constants depending only on the domain $Q$. Moreover, we
denote by $\mathcal{Q} (\cdot)$ or $\mathcal{Q}_i(\cdot)$, $i\in \mathbb{N}$, a continuous, nonnegative monotone increasing function depending on structural quantities, but independent of $\beta$. The constants $C$ or the functions $\mathcal{Q}(\cdot)$ may vary from
line to line and even within the same line. Any further dependence will be explicitly pointed out if necessary.

\subsection{Main result}

Now we state the main result of this paper.
\bt \label{main} Suppose that $n\leq 3$ and $\beta_0>0$ is an arbitrary but fixed constant. For each $\beta\in [0,\beta_0]$ and $M, M'>0$, there exists an exponential attractor $\mathcal{M}_\beta$ for the semigroup $S_\beta(t)$ defined by the global energy solutions to problem \eqref{e1}--\eqref{e2} if $\beta>0$ or problem \eqref{pfc}--\eqref{pfc2} if $\beta=0$ (cf. Theorems \ref{exe} and \ref{exe0} below) on the phase space
$$\mathcal{X}^{M,M'}_0=\{(u,v)\in \mathbb{X}_0^\beta: \ |\beta\langle v\rangle+\langle u\rangle|\leq M, \ |\langle v\rangle|\leq M'\},$$
which satisfies the following properties:

(P1) $\mathcal{M}_\beta$ is positively invariant, bounded in $\mathbb{X}_3^\beta$ and $\mathbb{X}_0^{\beta_0}$ with bounds that may depend on $\beta_0$ but are independent of $\beta$.

(P2) The rate of attraction is uniformly exponential: for every bounded set $B\in \mathcal{X}^{M,M'}_0$, there exist $K_B > 0$ and $\gamma_B> 0$
independent of $\beta$ such that
\be
{\rm dist}_{\mathbb{X}_0^\beta}(S(t)B, \mathcal{M}_\beta)\leq K_{B}e^{-\gamma_B t},\quad \forall\, t\geq 0,\label{exp}
\ee
where
${\rm dist}_{\mathbb{X}_0^\beta}$ denotes the Hausdorff semidistance of sets with respect to the $\mathbb{X}_0^\beta$-metric.

(P3) The fractal dimension of $\mathcal{M}_\beta$ in $\mathbb{X}_0^\beta$ is uniformly bounded with respect to $\beta$, that is, there exists a positive constant $K$ independent of $\beta$, such that ${\rm dim}_{\mathbb{X}_0^\beta} \mathcal{M}_\beta\leq K$.

(P4) The map $\beta\to \mathcal{M}_\beta$ $(\beta\in [0,\beta_0])$ is H\"older continuous in $\beta$, namely, it holds
\be
{\rm dist}_{\mathbb{X}_0^{\beta_1}}^{{\rm sym}}(\mathcal{M}_{\beta_1}, \mathcal{M}_{\beta_2})\leq C(\beta_1-\beta_2)^\frac16, \quad \text{for}\ \ 0\leq \beta_2<\beta_1\leq \beta_0.\non
\ee
\et
%

\br
       For the sake of simplicity, in this paper we only treat the nonlinearity $f$ of the physically relevant form \eqref{f}. Moreover, in the subsequent analysis, we do not have to impose the restriction $\alpha:=1-\epsilon  >0$ as in \cite{WW10, WW11, WWL09}.
       Actually, our results hold for more general (possibly non-convex) nonlinearities.
       For instance, we can take $ f\in C^{2,1}_{loc}(\mathbb{R})$ such that
      $$ f(0)=0,\quad  \liminf_{|s|\to+\infty} f'(s)>0,\quad \liminf_{|s|\to+\infty}\frac{f(s)}{s}=+\infty.$$
      We note that the physically relevant nonlinearity  $f(y)= y^3+(1-\epsilon)y$ with $\epsilon\in \mathbb{R}$ fulfills the above assumptions.
 \er
\br
We note that the choice of periodic boundary
conditions is realistic since in the crystalline materials the patterns of the nanostructures statistically repeat
throughout the domain, which is much larger than the length-scales of atoms. Nevertheless, our theoretical results also hold for homogeneous Neumann boundary conditions, as well as for mixed periodic-homogeneous Neumann boundary conditions. Of course, there is no loss of generality in choosing the unit period for any spatial direction.
\er

 \section{\textit{A priori} dissipative estimates}
 \label{estimates}
\setcounter{equation}{0}
\noindent
In this section we derive some \textit{a priori} dissipative estimates for the solutions to problem \eqref{e1}--\eqref{e2}, which are uniform with respect to $\beta$. The subsequent calculations are performed formally. However, they can be justified by working within a suitable Faedo--Galerkin approximation scheme (cf. \cite{GW3}) and then passing to the limit.

We first recall how the PFC equation \eqref{pfc} and the MPFC equation \eqref{e1} behave with respect to the (total) mass conservation.
Integrating the PFC equation \eqref{pfc} with respect to time, we see that the following mass conservation property holds
\be
\langle\phi(t)\rangle=\langle\phi_0\rangle,\quad \forall\, t\geq 0.\label{conpfc}
\ee
However, the mass conservation may fail for the MPFC equation \eqref{e1}. Nevertheless, it still obeys a kind of conservation law, namely,
 \be
 \beta\langle\phi_t(t)\rangle+\langle\phi(t)\rangle=\beta\langle\phi_1\rangle+\langle\phi_0\rangle, \quad \forall\, t\geq 0,\label{conODE2}
 \ee
 so that (cf. \cite{GW3})
 \be
 \langle\phi_t(t)\rangle=\langle\phi_1\rangle e^{-\frac{t}{\beta}},\quad  \langle\phi(t)\rangle = \beta\langle\phi_1\rangle+\langle\phi_0\rangle-\beta \langle\phi_1\rangle e^{-\frac{t}{\beta}},\quad \forall\, t\geq 0.\label{mde2}
 \ee
 \br
 It is easy to see that if $(\phi, \phi_t)$ is a solution to problem \eqref{e1}--\eqref{e2}   with initial data $\phi_1$ satisfying the zero-mean assumption $\langle\phi_1\rangle=0$, then $\langle \phi_t(t)\rangle=0$ and, in particular, the mass conservation  $\langle \phi(t)\rangle=\langle\phi_0\rangle$ holds  for all $t\geq 0$ (cf. \cite{WW10, WW11}).
 \er

 First, we can show the uniform dissipative estimate in $\mathbb{X}_0^\beta$:

\bl\label{es}
For any $\beta\in (0, \beta_0]$, suppose that $(\phi, \phi_t)$ is a regular solution to problem \eqref{e1}--\eqref{e2}. Then the following dissipative estimate holds:
 \be
 \|\phi(t)\|_{2}^2+ \beta\|\overline{ \phi_t}(t)\|_{-1}^2  \leq\mathcal{Q}(\|\phi_0\|_{2}, \|\overline{ \phi_1}\|_{-1}) e^{-\rho_0 t}+\rho_0', \quad \forall\, t\geq 0,\label{disa1}
 \ee
 where the positive constants $\rho_0, \rho_0'$ may depend on $\beta_0$, $\epsilon$,  $\langle\phi_1\rangle $, $\langle\phi_0\rangle$ and $|Q|$, but are independent of $\|\phi_0\|_{2}$, $\|\overline{ \phi_1}\|_{-1}$, the parameter $\beta$ and time $t$.
\el
\begin{proof}
We rewrite the equation \eqref{e1} into the following form
\be
\beta \phi_{tt}+\phi_t=-A_0 (\Delta ^2 \phi +2\Delta \phi+f(\phi)-\langle f(\phi)\rangle). \label{e1a}
\ee
Testing \eqref{e1a} by $A_0^{-1} \overline{\phi_t}(t)+\eta_1 A_0^{-1} \overline{\phi}(t)$ with $\eta_1>0$ being a small constant  to be determined later, we get
\be
\frac{d}{dt}\mathcal{Y}_1(t)+\mathcal{D}_1(t)\leq \mathcal{R}_1(t),\label{d3}
\ee
where
\bea
\mathcal{Y}_1&=& \frac{\beta}{2}\|\overline{\phi_t}\|_{-1}^2+ E(\phi)+  \eta_1 \beta (\overline{\phi_t}, \overline \phi)_{-1}+\frac{\eta_1}{2}\|\overline \phi\|_{-1}^2,\nonumber\\
\mathcal{D}_1&=& (1-\eta_1\beta)\|\overline{\phi_t}\|_{-1}^2+\eta_1 \|\Delta \phi\|^2+\eta \int_{Q} f(\phi) (\phi-(\beta\langle\phi_1\rangle+\langle\phi_0\rangle)) dx,\non\\
\mathcal{R}_1&=& (1-\eta_1\beta)\langle\phi_1\rangle e^{-\frac{t}{\beta}}\int_{Q} f(\phi) dx+2\eta_1 \|\nabla \phi\|^2.\non
\eea
Using integration by parts and the Cauchy--Schwarz inequality, we get
\be
\|\nabla \phi\|^2 \leq \|\Delta \phi\|\|\phi\|\leq \frac14\|\Delta \phi\|^2+ \|\phi\|^2.\label{intpo}
\ee
Recalling \eqref{F}, we easily see that $F$ is uniformly bounded from below
\be
F(\phi)\geq \frac14 \Big(\phi^2 + (1-\epsilon)\Big)^2-\frac14 (1-\epsilon)^2\geq -\frac14 (1-\epsilon)^2.\label{Fb}
\ee
Besides, we infer from  the Sobolev embedding $H^2(Q)\hookrightarrow L^\infty(Q)$ $(n\leq 3$) and Young's inequality that
\be
C(\|\phi\|_{2})\geq E(\phi)\geq \frac14\|\Delta \phi\|^2+ \|\phi\|^2-C_1.\label{esE}
\ee
On the other hand, the Cauchy--Schwarz inequality yields
\be
\eta_1 \beta |(\overline{\phi_t}, \overline\phi)_{-1}|\leq \frac{\beta}{4}\|\overline{\phi_t}\|_{-1}^2+\eta_1^2\beta\|\overline \phi\|_{-1}^2. \label{cs1}
\ee
As a result, for $\eta_1\in (0, \frac{1}{2\beta_0})$, we obtain
\be
C(\|\overline{ \phi_t}\|_{-1}, \|\phi\|_{H^2})\geq \mathcal{Y}_1(t)\geq \frac{\beta}{4}\|\overline{\phi_t}\|_{-1}^2+\frac14\|\Delta \phi\|^2+ \|\phi\|^2-C_1.\label{Y1}
\ee
Next, on account of \eqref{f} and \eqref{F}, we get
\be
\|\phi\|^2\leq \xi_1 \int_{Q} F(\phi) dx +C_{\xi_1},\quad \int_Q|f(\phi)| dx\leq \xi_2 \int_{Q} F(\phi) dx+C_{\xi_2},\non
\ee
and
\bea
&&
\int_{Q} f(\phi) (\phi-(\beta\langle\phi_1\rangle+\langle\phi_0\rangle)) dx\non \\
&\geq& -C_2\|\phi-(\beta\langle\phi_1\rangle+\langle\phi_0\rangle)\|^2 +C_3\int_{Q} F(\phi) dx-C_4\non\\
&\geq & -2C_2\|\phi\|^2 +C_3\int_{Q} F(\phi) dx-(C_4+2C_2(\beta\langle\phi_1\rangle+\langle\phi_0\rangle)^2)\non\\
&\geq& (-2C_2\xi_1+C_3)\int_{Q} F(\phi) dx-(C_4+2C_2(\beta\langle\phi_1\rangle+\langle\phi_0\rangle)^2+C_{\xi_1}).\non
\eea
Let $\kappa_1>0$ be a small constant to be chosen later. Then we have
\bea
&&\mathcal{D}_1(t)-\kappa_1 \mathcal{Y}_1(t)\non\\
&=& \left(1-\eta_1\beta-\frac{\kappa_1}{2}\beta\right)\|\overline{\phi}_t\|_{-1}^2+\left(\eta_1-\frac{\kappa_1}{2}\right)\|\Delta \phi\|^2+\kappa_1\|\nabla \phi\|^2\non\\
&&+\eta_1 \int_{Q} f(\phi) (\phi-M) dx-\kappa_1\int_Q F(\phi) dx-\frac{\eta_1\kappa_1}{2}\|\phi\|_{-1}^2\non\\
&\geq& \left(1-\eta_1\beta-\frac{\kappa_1}{2}\beta\right)\|\overline{\phi}_t\|_{-1}^2+\left(\eta_1-\frac{\kappa_1}{2}\right)\|\Delta \phi\|^2+\kappa_1\|\nabla \phi\|^2\non\\
&& +(C_3\eta_1-2C_2\xi\eta_1-\kappa_1)\int_Q F(\phi) dx-\eta_1\left(C_4+2C_2(\beta\langle\phi_1\rangle+\langle\phi_0\rangle)^2+C_\xi\right).\non
\eea
Next, for $\xi_2>0$, using the Young inequality and \eqref{intpo}, we can see that
\bea
\mathcal{R}_1(t)
&\leq& |\langle\phi_1\rangle e^{-\frac{t}{\beta}}|\int_{Q} |f(\phi)| dx+2\eta_1 \|\nabla \phi\|^2\non\\
&\leq& |\langle\phi_1\rangle|\left(\xi_2 \int_{Q} F(\phi)dx + C_{\xi_2}\right) + \frac{\xi_2\eta_1}{2}\|\Delta \phi\|^2+ \frac{2\eta_1}{\xi_2} \|\phi\|^2\non\\
&\leq& \left(|\langle\phi_1\rangle|\xi_2+ \frac{2\eta_1\xi_1}{\xi_2}\right)\int_{Q} F(\phi)dx +\frac{\xi_2\eta_1}{2}\|\Delta \phi\|^2+ C_{\xi_2}|\langle\phi_1\rangle| + \frac{\xi_2\eta_1}{2}C_{\xi_1}.\non
\eea
Choosing $\eta_1\in (0, \frac{1}{2\beta_0})$ and $\kappa_1, \xi_1, \xi_2 >0$ satisfying
\bea
&&1-\eta_1\beta_0-\frac{\kappa_1}{2}\beta_0\geq 0, \quad \eta_1-\frac{\kappa_1}{2}-\frac{\xi_2\eta_1}{2}\geq 0, \non\\ &&C_3\eta_1-2C_2\xi_1\eta_1-\kappa_1-|\langle\phi_1\rangle|\xi_2-\frac{2\eta_1\xi_1}{\xi_2}\geq 0,\non
\eea
we have, for all $\beta\in (0,\beta_0]$,
\bea
 && \mathcal{D}_1(t)-\mathcal{R}_1(t)
  \non\\
  &\geq&
  \kappa_1\mathcal{Y}_1(t) -\eta_1\left(C_4+2C_2(\beta\langle\phi_1\rangle+\langle\phi_0\rangle)^2+C_{\xi_1}\right)- C_{\xi_2}|\langle\phi_1\rangle| -\frac{\xi_2\eta_1}{2}C_{\xi_1}.\non
\eea
Collecting the above estimates together,
we infer from inequalities \eqref{d3} and \eqref{Fb} that
\be
\frac{d}{dt}\mathcal{Y}_1(t)+\kappa_1\mathcal{Y}_1(t) \leq C_5,\label{d3a}
\ee
where $C_5$ may depend on $Q, \beta_0, \epsilon$, $|\langle\phi_0\rangle|$ and $|\langle\phi_1\rangle|$. Then  we get
\be
\mathcal{Y}_1(t)\leq \mathcal{Y}_1(0)e^{-\kappa_1 t}+\frac{C_{5}}{\kappa_1},\quad  \forall \, t\geq 0,\non
\ee
which together with \eqref{Y1} yields our conclusion. The proof is complete.
\end{proof}
Then we can proceed to prove the higher-order dissipative estimates in $\mathbb{X}_j^\beta$:
\bl\label{es1}
For any $\beta\in (0, \beta_0]$, suppose $(\phi,\phi_t)$ is a regular solution to problem \eqref{e1}--\eqref{e2}. Then the following estimate holds for $t\geq 0$:
 \bea
  && \|(\phi(t), \phi_t(t))\|_{\mathbb{X}_j^\beta}^2
 \leq \mathcal{Q}(\|(\phi_0, \phi_1)\|_{\mathbb{X}_j^\beta}) e^{-\rho_j t}+\int_0^te^{-\rho_j'(t-s)}\mathcal{Q}(\|\phi(s)\|_{j+1}) ds,\label{disa1a}\\
 && \sup_{t\geq 0}\int_t^{t+1} \|\phi_t(s)\|_{j-1}^2 ds\leq \mathcal{Q}(\|(\phi_0, \phi_1)\|_{\mathbb{X}_j^\beta}),\label{esphit}
 \eea
 where the positive constants $\rho_j, \rho_j'$ ($j=1, 2, 3$) may depend on $\beta_0$, $\epsilon$,  $\langle\phi_1\rangle $, $\langle\phi_0\rangle$ and $|Q|$, but independent of the $\mathbb{X}_j^\beta$-norm of initial data, $\beta$ and time $t$.
\el
\begin{proof} We first consider the case $j=1$. Testing \eqref{e1} by $ \phi_t(t)+\eta_2 \phi(t)$ with $\eta_2\in (0, \frac{1}{2\beta_0})$ to be determined later,  we obtain
\be
\frac{d}{dt}\mathcal{Y}_2(t)+\mathcal{D}_2(t)\leq \mathcal{R}_2(t),\label{d3h}
\ee
where
\bea
\mathcal{Y}_2 &=&\frac{\beta}{2}\|\phi_t\|^2+\frac12\|\nabla \Delta \phi\|^2-\|\Delta \phi\|^2
 +\beta\eta_2\int_{Q} \phi_t\phi dx +\frac{\eta_2}{2}\|\phi\|^2,\non\\
\mathcal{D}_2 &=& (1-\eta_2\beta)\|\phi_t\|^2+\eta_2\|\nabla \Delta \phi\|^2,\non\\
\mathcal{R}_2 &=&  \int_{Q}\Delta f (\phi)\phi_t dx+2\eta_2\|\Delta \phi\|^2+\eta_2\int_{Q}f(\phi)\Delta \phi dx.\non
\eea
The uniform dissipative estimate \eqref{disa1} together with the Sobolev embedding $H^2(Q)\hookrightarrow L^\infty(Q)$ ($n\leq 3$) yields the (uniform) global boundedness of $\phi$, that is,
  $$
  \|\phi(t)\|^2_{L^\infty}\leq C\mathcal{Q}_1(\|\phi_0\|_{2}, \|\overline{ \phi_1}\|_{-1}) e^{-\rho_1 t}+C\rho_2, \quad \forall\, t\geq 0.
  $$
As a consequence, thanks to \eqref{f} and to the Young inequality, we have
\be
\mathcal{Q}_2(\|\phi\|_{3}, \|\overline{ \phi_t}\|)\geq \mathcal{Y}_2(t)\geq \frac14(\beta\|\phi_t\|^2+\|\nabla \Delta\phi\|^2)-C\|\phi\|_2^2.\label{eY2}
\ee
Using the Sobolev embeddings once more, we are able to estimate the terms in $\mathcal{R}_2$ as follows
    \bea
    \int_{Q}\Delta f (\phi)\phi_t dx
    &\leq& \frac12\|\phi_t\|^2+\frac12\|\Delta f(\phi)\|^2\non\\
    &\leq& \frac12\|\phi_t\|^2+\mathcal{Q}(\|\phi\|_2),\non
    \eea
    \be
    2\eta_2\|\Delta \phi\|^2+\eta_2\int_{Q}f(\phi)\Delta \phi dx\leq c\eta_2(\|\phi\|_2^2+\|f(\phi)\|^2)\leq \mathcal{Q}(\|\phi\|_2).\non
    \ee
    Then for $\kappa_2>0$ to be determined later, we have
\bea
&& \mathcal{D}_2(t)-\kappa_2\mathcal{Y}_2(t)-\mathcal{R}_2(t)\non\\
&\geq&  \Big(\frac12-\eta_2\beta-\frac{\kappa_2\beta}{4}\Big)\|\phi_t\|^2+\Big(\eta_2-\frac{\kappa_2}{4}\Big)\|\nabla \Delta \phi\|^2-\mathcal{Q}(\|\phi\|_2).\non
\eea
    Choosing $\eta_2, \kappa_2>0$ satisfying
    \be
    \frac12-\eta_2\beta_0-\frac{\kappa_2\beta_0}{4}\geq \frac14, \quad \eta_2-\frac{\kappa_2}{4}\geq 0,\non
    \ee
    we deduce that the following inequality holds
    \be
    \frac{d}{dt}\mathcal{Y}_2(t)+ \frac14\|\phi_t(t)\|^2+\kappa_2\mathcal{Y}_2(t)\leq \mathcal{Q}(\|\phi\|_2),\non
    \ee
    where $\mathcal{Q}$ is a continuous monotone function satisfying $\mathcal{Q}(0)=0$. Applying the Gronwall inequality and the dissipative estimate \eqref{es}, we conclude that
    \be
    \mathcal{Y}_2(t) \leq
    \mathcal{Y}_2(0)e^{-\kappa_2t}+\int_0^te^{-\kappa_2(t-s)}\mathcal{Q}(\|\phi(s)\|_2) ds,\non
    \ee
which combined with \eqref{es} and \eqref{eY2} easily yields the dissipative estimates \eqref{disa1a} and \eqref{esphit}.

The higher-order dissipative estimates for $j=2, 3$ can be obtained in a similar way. For instance, for $j=2$, we can test the equation \eqref{e1} by $-\Delta \phi_t(t)-\eta_3 \Delta \phi(t)$ with suitably small constant $\eta_3>0$. For $j=3$, we apply the operator $\Delta$ to \eqref{e1} and test the resultant by  $\Delta \phi_t(t)+\eta_4 \Delta \phi(t)$ for some $\eta_4>0$ sufficiently small.  Repeating the above argument and making use of the dissipative estimates obtained in the previous step for $j-1$, we can reach our conclusion. The proof is left to the interested reader and thus is omitted here.
\end{proof}

We conclude this section with a bound on the second order time derivative $\phi_{tt}$, which will be useful in the last section.

\bc\label{phitt}
For any $\beta\in (0, \beta_0]$, suppose that $(\phi,\phi_t)$ is a regular solution to problem \eqref{e1}--\eqref{e2}. Then we have \be
\sup_{t\geq 0} \int_t^{t+1} \beta \|\phi_{tt}(\tau)\|_{-1}^2 d \tau \leq  \mathcal{Q}(\|(\phi_0, \phi_1)\|_{\mathbb{X}_3^\beta}).\label{esphitt}
\ee
\ec
\begin{proof}
Testing \eqref{e1a} by $A_0\overline{\phi_{tt}}(t)$ and integrating by parts, we get
\bea
&& \frac{d}{dt}\left[\frac12 \|\overline{\phi_t}\|_{-1}^2+(\Delta^2\phi, \overline{\phi_t}), +2(\Delta \phi, \overline{\phi_t})+ (f(\phi), \overline{\phi_t})\right]+\beta\|\overline{\phi_{tt}}\|_{-1}^2\non\\
&=& (\Delta \phi_t, \Delta \overline{\phi_t})-2(\nabla  \phi_t, \nabla \overline{\phi_t})+(f'(\phi)\phi_t, \overline{\phi_t})\non\\
&\leq& C\|\phi_t\|_2^2+C(\|\phi\|_2)\|\phi_t\|^2.\non
\eea
Integrating the above inequality with respect to time and using the uniform estimates \eqref{disa1a} and \eqref{esphit} with $j=3$, we easily infer \eqref{esphitt}. The proof is complete.
\end{proof}

\section{The dissipative dynamical system}
\label{dissdyn}
\setcounter{equation}{0}
\subsection{Well-posedness}
Based on the uniform estimate \eqref{disa1}, existence and uniqueness of global energy solutions to problem \eqref{e1}--\eqref{e2} has been proven in \cite[Theorem 4.1]{GW3} via a suitable Galerkin approximation for arbitrary but fixed $\beta>0$. Namely, there holds
\begin{theorem}\label{exe}
  For any $\beta >0$ and initial data $(\phi_0, \phi_1)\in \mathbb{X}_0$, the MPFC equation \eqref{e1}--\eqref{e2} admits a unique global energy solution $(\phi, \phi_t)$ such that
$$\phi\in C^2([0,T];H^{-4}_p(Q))\cap C^1([0,T]; H^{-1}_p(Q))\cap C([0,T]; H^2_p(Q))$$
satisfying
\bea&&
A_0^{-1}(\beta \phi_{tt}+\phi_t)+ \Delta ^2 \phi +2\Delta \phi+f(\phi)-\langle f(\phi)\rangle
=0,\non\\
 && \qquad \qquad \qquad \qquad\qquad \mbox{in}\  D(A_0^{-1}), \quad \text{a.e. in}\ (0,T), \label{e1ae}\\
&&
\phi|_{t=0}=\phi_0 \ \mbox{in} \ H^2_p(Q),\quad \phi_t|_{t=0}=\phi_1 \ \mbox{in} \ H^{-1}_p(Q).\label{e2ae}
\eea
Moreover, the following energy identity holds for all $ s, \,t\in [0,T]$ with $s<t$
\be
\mathcal{E}(t)=\mathcal{E}(s)-\int_s^t \|\overline{\phi_t}(\tau)\|_{-1}^2 d\tau+\int_s^t\langle\phi_1\rangle e^{-\frac{\tau}{\beta}}\int_Q f(\phi(\tau))dx d\tau. \label{enereq}
\ee
 \end{theorem}
From the energy identity \eqref{enereq}, we are able to check that the solution $(\phi, \phi_t)$ is indeed uniformly Lipschitz continuous on bounded balls of $\mathbb{X}_0^\beta$, for any fixed $T\geq 0$:

\begin{corollary} Suppose $\beta\in (0,\beta_0]$.
Consider two pairs of initial data $(\phi_{0j}, \phi_{1j})$ $(j=1,2)$ in a bounded set of $\mathbb{X}_0^\beta$, $\phi_j$ being the solution to problem \eqref{e1}--\eqref{e2} corresponding to the initial datum $(\phi_{0j}, \phi_{1j})$.
The following continuous dependence estimate holds
\bea
&&
\|(\phi_1-\phi_2, \phi_{1t}-\phi_{2t})(t)\|_{\mathbb{X}_0^\beta}^2+ \int_0^t \Vert ( \phi_{1t}-\phi_{2t}) (\tau)\Vert_{-1}^2 d\tau\non\\
&\leq& L_1 e^{L_2T} \|(\phi_{10}-\phi_{20}, \phi_{11}-\phi_{21})\|_{\mathbb{X}_0^\beta}^2,\quad \forall\, t\in [0,T].\label{LipX0}
\eea
where $L_1$, $L_2$ are positive constants depending on  $\|\phi_{0j}\|_2$, $\|\phi_{1j}\|_{-1}$ as well as on $\beta_0$, $\epsilon$, $Q$ and $f$, but independent of $\beta$ and $t$.
\end{corollary}
\begin{proof}
Denote the difference of solutions by $\tilde{\phi}=\phi_1-\phi_2$. In \cite{GW3}, we have shown that
\bea
&&{\cal E}_0(\tilde{\phi}(t),\tilde{\phi}_t(t))+ |\langle\tilde{\phi}(t)\rangle|^2+\int_0^t \Vert \overline{ \tilde{\phi}_t}(\tau)\Vert_{-1}^2 d\tau \non\\
&\leq& {\cal E}_0(\tilde{\phi}_0,\tilde{\phi}_1)+C|\langle \tilde{\phi}(0)\rangle|^2+ C \int_0^t {\cal E}_0(\tilde{\phi}(\tau),\tilde{\phi}_t(\tau)) + |\langle\tilde{\phi}(\tau)\rangle|^2 d\tau,\non
\eea
where
\be
{\cal E}_0(\psi,\psi_t):= \frac{\beta}{2}\Vert \overline{\psi_t}\Vert_{-1}^2 +{1\over 2} \| \Delta\psi\|^2 - \| \nabla\psi\|^2 + {\Lambda\over 2}\| \overline{\psi} \|_{-1}^2
 \label{enlin1}
\ee
 and $\Lambda>0$ is a sufficiently large constant.

 Thus, using the fact $$\langle \tilde{\phi}(t)\rangle = \beta\langle \tilde{\phi}_1\rangle+\langle \tilde{\phi}_0\rangle-\beta \langle \tilde{\phi}_1\rangle e^{-\frac{t}{\beta}},$$
we conclude \eqref{LipX0} from by the Gronwall inequality. The proof is complete.
\end{proof}

The previous results imply similar results hold for the limiting case $\beta=0$ (with an even simpler proof):
\begin{theorem}\label{exe0}
  For any initial data $\phi_0\in H_p^2(Q)$, the PFC equation \eqref{pfc}--\eqref{pfc2} admits a unique global energy solution $$\phi\in  C([0,T]; H^2_p(Q))\cap H^1(0,T; H^{-1}_p(Q)).$$
   Besides, the following continuous dependence estimate holds for all $t\in [0,T]$
   \be
\|\phi_1(t)-\phi_2(t)\|_{2}^2+ \int_0^t \Vert ( \phi_{1t}-\phi_{2t}) (\tau)\Vert_{-1}^2 d\tau\leq L_1 e^{L_2T} \|\phi_{10}-\phi_{20}\|_{2}^2 ,\label{LipX00}
\ee
where $L_1$, $L_2$ are positive constants depending on $\|\phi_0\|_2$ as well as on $\epsilon$, $Q$ and $f$, but independent of $t$. Moreover,  $\phi\in C^\infty((0,+\infty)\times Q)$.
 \end{theorem}

 \textit{The Associated Semigroups}. We can now associate with problem \eqref{e1}--\eqref{e2}  a family of strongly continuous semigroups $S_\beta(t):\mathbb{X}^\beta_0\to \mathbb{X}^\beta_0$, $\beta\in (0, \beta_0]$,  by setting
  $$\mathbf{u}(t)=S_\beta(t)\mathbf{u}_0, \quad \forall \, t\geq 0,$$
where $\mathbf{u}(t)=S_\beta(t)\mathbf{u}_0$ is the unique energy solution given by Theorem \ref{exe} corresponding to the initial data $\mathbf{u}_0=(\phi_0, \phi_1)\in \mathbb{X}_0^\beta$.

Similarly, for the limiting case $\beta=0$, we can define the strongly continuous semigroup $S_0(t): \mathbb{X}^0_0\to \mathbb{X}^0_0$ by setting $$\mathbf{u}(t)=S_0(t)\mathbf{u}_0, \quad \forall \, t\geq 0,$$
where $\mathbf{u}_0=(\phi_0, 0)\in \mathbb{X}_0^0$, $\mathbf{u}(t)=(\phi(t), 0)$ and $\phi(t)$ is the unique energy solution given by Theorem \ref{exe0} corresponding to the initial datum $\phi_0$.

\subsection{Absorbing sets}

Thanks to  Lemma \ref{es}, we can state some dissipative properties of the dynamical system $(S_\beta(t), \mathbb{X}_0^\beta)$ defined on a suitable phase space. Recalling the conservative identities \eqref{conpfc} and \eqref{conODE2}, we have to work on the following (closed) subset of $\mathbb{X}_0^\beta$:
   $$\mathcal{X}^{M,M'}_0=\{(u,v)\in \mathbb{X}_0^\beta: \ |\beta\langle v\rangle+\langle u\rangle|\leq M, \ |\langle v\rangle|\leq M'\}.$$
Then we conclude the existence  of a bounded absorbing set in $\mathcal{X}^{M,M'}_0$ from the dissipative estimate \eqref{disa1} (cf. Lemma \ref{es}),
namely,

\begin{proposition}\label{abs}
Let $\beta\in [0,\beta_0]$. We indicate by $\mathcal{B}_0(R)$ a generic ball in $\mathcal{X}^{M,M'}_0$ of radius $R$. There exists $R_0>0$ that may depend on $Q, \epsilon, f, \beta_0, M, M'$ but independent of $\beta$, such that for all $R>0$, there is a $t_R>0$ such that
\be
S_\beta(t)\mathcal{B}_0(R) \subset \mathcal{B}_0(R_0), \quad \forall\, t\geq t_R.\label{R0}
\ee
\end{proposition}
\br
Since $\mathcal{B}_0(R_0)$ is also a bounded set in $\mathcal{X}^{M,M'}_0$, then there exists $t_{R_0}>0$ such that
\be
S_\beta(t)\mathcal{B}_0(R_0)\subseteq \mathcal{B}_0(R_0), \quad \forall\,t\geq t_{R_0}.\non
\ee
For $\beta \in [0,\beta_0]$, we note that $t_{R_0}$ may depend on $\beta_0$ but is independent of $\beta$. Set for any $\beta \in [0,\beta_0]$,
\be
\mathbb{Y}_0^\beta:=\overline{\bigcup_{t\in [0,t_{R_0}]}S_\beta(t)\mathcal{B}_0(R_0)}^{\mathbb{X}^\beta_0}.\label{spaceY0}
\ee
We can see that $\mathbb{Y}_0^\beta$ is a complete bounded metric space with respect to the metric of $\mathbb{X}_0^\beta$.
Moreover, the following properties holds:
\be
 S_\beta(t)\mathbb{Y}_0^\beta \subseteq \mathbb{Y}_0^\beta\quad \text{and}\quad \|S_\beta(t)\mathbf{u}_0\|_{\mathbb{X}_0^\beta}\leq C(R_0), \quad \forall\, t\geq 0, \ \ \forall\, \mathbf{u}_0\in \mathbb{Y}_0^\beta.\non
\ee
The constant $C(R_0)$ may depend on $\beta_0$ but is independent of $\beta$.
\er
In what follows, we shall work on the (positively invariant) phase space $\mathbb{Y}_0^\beta$ (cf. \eqref{spaceY0}). On the other hand, using the higher-order dissipative estimates in Lemma \ref{es1}, we are able to prove the existence of absorbing sets in more regular spaces, i.e.,
$$ \mathbb{Y}_i^\beta:=\mathbb{Y}_0^\beta\cap \mathbb{X}_i^\beta, \quad i=1,2,3.$$
Indeed, we have
\begin{proposition}\label{absh}
Let $\beta\in [0,\beta_0]$. We denote $\mathcal{B}_i(R)$ $(i=1,2,3)$ a generic ball in $\mathbb{Y}_i^\beta$ of radius $R>0$.
There exist $R_3\geq R_2\geq R_1\geq R_0$ ($R_0$ is given in Proposition  \ref{abs}), that may depend on $Q, \epsilon, f, \beta_0, M, M'$ but are independent of $\beta$, such that $\mathcal{B}_i(R_i)$  are absorbing sets in $\mathbb{Y}_i^\beta$ $(i=1,2,3)$, respectively. Namely, for any bounded set $B_i\in \mathbb{Y}_i^\beta$, there exists $t_{B_i}>0$ such that
\be
S_\beta(t)B_i\subset \mathcal{B}_i(R_i), \quad \forall\, t\geq t_{B_i}.\non
\ee
\end{proposition}
\begin{proof}
For any $(\phi_0, \phi_1)\in B_1$, we infer from the definition of $\mathbb{Y}_1^\beta$ that
\be
\|\phi(t)\|_2\leq C(R_0),\quad \forall\, t\geq 0.\non
\ee
It follows from inequality \eqref{disa1a} (with $j=1$) that
\bea
\|\phi(t)\|_{3}^2+ \beta\|\phi_t(t)\|^2
 &\leq& \mathcal{Q}(\|\phi_0\|_{3}, \|\phi_1\|) e^{-\rho_1 t}+\int_0^te^{-\rho_1'(t-s)}\mathcal{Q}(\|\phi(s)\|_2) ds,\non\\
 &\leq& \mathcal{Q}(\|\phi_0\|_{3}, \|\phi_1\|) e^{-\rho_1 t}+C(R_0)\int_0^te^{-\rho_1'(t-s)} ds\non\\
 &\leq& C(B_1) e^{-\rho_1 t}+C(R_0).\non
\eea
Then there exists $t_{B_1}>0$ such that
\be
\|(\phi(t), \phi_t(t)\|_{\mathbb{X}_1^\beta}\leq C(R_0), \quad\forall\, t \geq t_{B_1}.\label{B1}
\ee
 We just set $R_1=\max\{C(R_0), R_0\}$.

 Next, for any $B_2\subset\mathbb{Y}_2^\beta$, we can find $B_1\subset \mathbb{Y}_1^\beta$ such that $B_2\subset B_1$. Using estimates \eqref{disa1a} ($j=2$) and \eqref{B1}, we have for $t\geq t_{B_1}$
\bea
\|\phi(t)\|_{4}^2+ \beta\|\phi_t(t)\|_1^2
 &\leq& \mathcal{Q}(\|\phi_0\|_{4}, \|\phi_1\|_1) e^{-\rho_2 t}+\int_0^te^{-\rho_2'(t-s)}\mathcal{Q}(\|\phi(s)\|_3) ds\non\\
 &\leq& C(B_2) e^{-\rho_2 t}+C(R_1),\non
\eea
which yields that there exists $t_{B_2}\geq t_{B_1}$ such that \be
\|(\phi(t), \phi_t(t)\|_{\mathbb{X}_2^\beta}\leq C(R_1), \quad\forall\, t \geq t_{B_2}.\label{B2}
\ee
Then we can set $R_2=\max\{C(R_1), R_1\}$.

In a similar manner, we can prove the absorbing property in $\mathbb{Y}_3^\beta$ from \eqref{disa1a} (with $j=3$) and estimate  \eqref{B2}. The proof is complete.
\end{proof}

\subsection{Attracting sets}

 A basic step in the existence proof of global or exponential attractors is to show certain (pre)compactness property of trajectories in the phase space \cite{MZ}.  For the limiting case $\beta=0$, the PFC equation \eqref{pfc} is a sixth-order parabolic equation whose solution is smooth for $t>0$. However, when $\beta>0$, we have to overcome difficulties arising from the hyperbolic-like nature of the MPFC equation \eqref{e1}.

 To this end, we try to find a proper decomposition of the semigroup $S_\beta(t)$ ($\beta>0$) into a uniformly asymptotically stable part and a compact part.
Let $(\phi, \phi_t)$ be the unique energy solution to problem \eqref{e1}--\eqref{e2}   given in Theorem \ref{exe}. We split this solution
into two parts, namely,
 $$  (\phi, \phi_t)(t)= (\phi^d, \phi^d_t)(t)+ (\phi^c, \phi^c_t)(t)$$
such that
 \bea
&& A_0^{-1}(\beta \phi^d_{tt}+\phi^d_t)+ \Delta^2\phi^d+ 2\Delta \phi^d+f_k(\phi^d)-\langle f_k(\phi^d)\rangle=0,\label{e1d}
\\
&&\phi^d|_{t=0}=\overline{{\phi}_0}(x),\quad \phi^d_t|_{t=0}=\overline{\phi_1}(x), \label{e2d}
\eea
 and
 \bea
&& A_0^{-1}(\beta \phi^c_{tt}+\phi^c_t)+\Delta^2 \phi^c+2\Delta \phi^c+f_k(\phi)-f_k(\phi-\phi^c)\non\\
&&\qquad -\langle f_k(\phi)\rangle+\langle f_k(\phi-\phi^c)\rangle =k\phi-k\langle\phi\rangle,\label{e1c}
\\
&&\phi^c|_{t=0}=\langle\phi_0(x)\rangle,\quad \phi^c_t|_{t=0}=\langle\phi_1(x)\rangle.\label{e2c}
\eea
Here, we have set
$$f_k(\phi):=f(\phi)+k\phi$$ with $k>0$ being a sufficiently large constant to be chosen later. In particular,
we require that the function $f_k(s)$ is monotone and nondecreasing in $\mathbb{R}$.

\bl\label{decay}
Suppose that $\beta\in (0, \beta_0]$ and the assumptions of Theorem \ref{exe} hold. For any $(\phi, \phi_1)\in \mathbb{Y}_0^\beta$, we have
\be
\|(\phi^d(t), \phi_t^d(t))\|_{\mathbb{X}^\beta_0}\leq C(R_0)e^{-\kappa_1 t}, \quad \forall\, t\geq 0,\label{decayx0}
\ee
where $\kappa_1>0$ is a small constant independent of $\beta$.
\el
\begin{proof}
For any positive constant $k$, the existence and uniqueness of a global energy solution $(\phi^d, \phi_t^d)$ to problem \eqref{e1d}--\eqref{e2d} easily follows from the same argument used to prove Theorem \ref{exe} (cf. \cite{GW3}). Moreover, due to the zero-mean assumption on the initial data \eqref{e2d}, we conclude that
$$
\langle\phi^d(t)\rangle=\langle\phi_t^d(t)\rangle=0, \quad \forall\, t\geq 0,
$$
which also yields that $\phi^d(t)=\overline{\phi^d}(t)$ and $ \phi^d_t(t)=\overline{\phi^d_t}(t)$.

Testing \eqref{e1d} by $\phi^d_t(t)+\eta_1 \phi^d(t)$, for some $\eta_1\in (0, \frac{1}{2\beta_0})$, we have
\be
\frac{d}{dt}\mathcal{Y}^d_1(t)+\mathcal{D}^d_1(t)\leq 0,\label{decayx03}
\ee
where
\bea
\mathcal{Y}^d_1&=& \frac{\beta}{2}\|\phi^d_t\|_{-1}^2+ \frac12\|\Delta \phi^d\|^2-\|\nabla \phi^d\|^2+ \int_{Q} F_k(\phi^d) dx\non\\
&&\quad +  \eta_1 \beta (\phi^d_t, \phi^d)_{-1}+\frac{\eta_1}{2}\| \phi^d\|_{-1}^2,
\nonumber\\
\mathcal{D}^d_1&=& (1-\eta_1\beta)\|\phi^d_t\|_{-1}^2+\eta_1 \|\Delta \phi^d\|^2-2\eta_1 \|\nabla \phi^d\|^2+\eta_1 \int_{Q} f_k(\phi^d) \phi^d dx,\non\\
&& \text{with} \quad F_k(\phi^d)=\frac{1-\epsilon+k}{2}(\phi^d)^2+\frac14 (\phi^d)^4.\non
\eea
We take $k>0$ sufficiently large and satisfying $1-\epsilon+k\geq 2\eta_1$ (it may depend on $Q$, $\beta_0$, $\epsilon$) such that
\be
C(\|\phi^d_t\|_{-1}, \|\phi^d\|_{H^2})\geq \mathcal{Y}^d_1(t)\geq \frac{\beta}{4}\|\phi^d_t\|_{-1}^2+\frac14\|\Delta \phi^d\|^2+ \frac{1}{2}\|\phi^d\|^2. \label{Yd1}
\ee
For $\kappa_1>0$, we find
\be
 \mathcal{D}^d_1(t)-\kappa_1\mathcal{Y}_1^d(t)\geq \Big(1-\eta_1\beta-\frac{\kappa_1\beta}{4}\Big) \|\phi^d_t\|_{-1}^2+\Big(\frac{\eta_1}{2}-\frac{\kappa_1}{4}\Big) \|\Delta \phi^d\|^2.\non
\ee
Finally, we take $\kappa_1>0$ satisfying
$$
1-\eta_1\beta_0-\frac{\kappa_1\beta_0}{4}\geq 0, \quad \eta_1-\frac{\kappa_1}{2}\geq 0.
$$
Then we infer from \eqref{decayx03} that
\be
\frac{d}{dt}\mathcal{Y}^d_1(t)+\kappa_1\mathcal{Y}^d_1(t)\leq 0,\label{decayx04}
\ee
which implies
\be
\mathcal{Y}^d_1(t)\leq \mathcal{Y}^d_1(0)e^{-\kappa_1 t}.\non
\ee
From \eqref{Yd1} we deduce that \eqref{decayx0} holds. The proof is complete.
\end{proof}

\bl\label{com}
Let the assumptions of Theorem \ref{exe} hold. For any $(\phi, \phi_1)\in \mathbb{Y}_0^\beta$, we have
\be
\|(\phi^c(t), \phi_t^c(t))\|_{\mathbb{X}_1^\beta}\leq C(R_0), \quad \forall\, t\geq 0.\label{comx0}
\ee
\el
\begin{proof}
Let the constant $k$ be the one chosen in Lemma \ref{decay}. For the initial data $(\phi_0, \phi_1)$ belonging to a bounded set in $\mathbb{Y}_0^\beta$, it follows from the uniform estimates \eqref{R0} and \eqref{decayx0} that
 \be
 \|(\phi^c(t), \phi_t^c(t))\|_{\mathbb{X}^\beta_0}\leq C(R_0), \quad \forall\, t\geq 0.\label{comx1}
 \ee
 \par Next, we prove the fact that $(\phi^c, \phi_t^c)$ is indeed more regular. We perform some higher-order calculations that can be justified rigorously by working within a proper Galerkin scheme as in \cite{GW3}.

Testing \eqref{e1c} by $A_0 \overline{\phi^c_t}(t)+\eta_2 A_0 \overline{\phi^c}(t)$, for some $\eta_2>0$, we get
 \be
\frac{d}{dt}\mathcal{Y}^c_1(t)+\mathcal{D}^c_1(t)\leq \mathcal{R}^c_1(t),\label{comx4}
\ee
where
\bea
\mathcal{Y}^c_1&=& \frac{\beta}{2}\|\overline{\phi^c_t}\|^2+\frac12
 \|\nabla \Delta \phi^c\|^2-\|\Delta \phi^c\|^2+\frac{k}{2}\|\nabla \phi^c\|^2+\eta_2\beta(\overline{\phi^c_t}, \overline{\phi^c})+\frac{\eta_2}{2}\|\overline{\phi^c}\|^2,\non\\
 \mathcal{D}^c_1&=&(1-\eta_2\beta)\|\overline{\phi^c_t}\|^2+\eta_2\|\nabla \Delta \phi^c\|^2-2\eta_2\|\Delta \phi^c\|^2+\eta_2k\|\nabla \phi^c\|^2,\non\\
 \mathcal{R}^c_1&=&\int_Q\Delta(f(\phi)-f(\phi-\phi^c))\overline{\phi^c_t} dx+\eta_2\int_Q(f(\phi)-f(\phi-\phi^c))\Delta \phi^c dx\non\\
 && \quad -k\int_Q \Delta \phi \overline{\phi^c_t} dx-\eta_2k\int_Q \phi \Delta \phi^c dx.\non
\eea
Arguing as before, for sufficiently large $k$ and small constants $\eta_2$, $\kappa_2$ (which may depend on $\beta_0$ but not on $\beta$), we can easily see that
 \be
 \mathcal{D}^c_1(t)\geq \frac12\|\overline{\phi_t^c}\|^2+\kappa_2\mathcal{Y}^c_1(t)\geq C(\beta\|\overline{\phi^c_t}\|^2+
 \|\phi^c\|_{3}^2),\label{comx7}
 \ee
  Due to the Sobolev embedding $H^2(Q) \hookrightarrow L^\infty(Q)$ ($n\leq 3$), the remainder term $\mathcal{R}^c_1$ can be estimated by using the uniform estimates \eqref{R0} and \eqref{comx1} as follows
 \bea
 \mathcal{R}_1^c(t)&\leq&
  \|\overline{\phi^c_t} \| \|\Delta(f(\phi)-f(\phi-\phi^c))\|+\eta_2\|f(\phi)-f(\phi-\phi^c)\|\|\Delta \phi^c\|\non\\
 && + k\|\Delta \phi\|\|\overline{\phi^c_t} \|+\eta_2 k\|\phi\|\|\Delta \phi^c\|\non\\
 &\leq& \frac12\|\overline{\phi_t^c}\|^2+ \|\Delta(f(\phi)-f(\phi-\phi^c))\|^2+ k^2\|\Delta \phi\|^2\non\\
 &&+\eta_2\|f(\phi)-f(\phi-\phi^c)\|\|\Delta \phi^c\| +\eta_2 k\|\phi\|\|\Delta \phi^c\|\non\\
 &\leq& \frac12\|\overline{\phi_t^c}\|^2+ C(\|\phi\|_{2}, \|\phi^c\|_{2})\non\\
 &\leq&  \frac12\|\overline{\phi_t^c}\|^2+ C(\|\phi_0\|_2, \|\phi_1\|_{-1}, \beta_0).\non
 \eea
 The above estimate combined with \eqref{comx4} and \eqref{comx7} yields
 \be
\frac{d}{dt}\mathcal{Y}^c_1(t)+\kappa_2\mathcal{Y}^c_1(t)\leq C(R_0).\label{comx5}
\ee
 As a result, we find
 \be
 \mathcal{Y}^c_1(t)\leq \mathcal{Y}^c_1(0)e^{-\kappa_2 t}+ \frac{C(R_0) }{\kappa_2},\quad \forall \, t\geq 0.\label{comx6}
 \ee
 On the other hand, the choice of initial data \eqref{e2c} indicates that $\mathcal{Y}_1^c(0)=0$.
 Then, from \eqref{comx7} and \eqref{comx6}  we conclude that \eqref{comx0} holds. The proof is complete.
\end{proof}

We then deduce from Lemmas \ref{decay} and \ref{com} the existence of a (bounded) attracting set for $S_\beta(t)$ in $\mathbb{Y}_1^\beta$.

\bp\label{atset1}
 For all $\beta\in (0, \beta_0]$, there exists $\mathcal{B}_1(K_1)\subset \mathbb{Y}_1^\beta$ such that
\be
{\rm dist}_{\mathbb{X}_0^\beta}(S_\beta(t)\mathbb{Y}_0^\beta, \mathcal{B}_1(K_1))\leq C(R_0) e^{-\zeta_1t}, \quad \, \forall\,t\geq 0,\non
\ee
where $K_1>0$ and $\zeta_1>0$ are independent of $\beta$.
\ep

Using the same decomposition of the trajectory $(\phi, \phi_t)$, we can further deduce the existence of attracting sets that are bounded in higher-order spaces.

\bp\label{atset2}
For all $\beta\in (0, \beta_0]$, there exist a bounded set $\mathcal{B}_i(K_i)\subset \mathbb{Y}_i^\beta$ $(i=2,3)$ such that for any bounded set $B_i \in \mathbb{Y}_{i-1}^\beta$, there holds
\be
{\rm dist}_{\mathbb{X}_0^\beta}(S_\beta(t)B_i, \mathcal{B}_i(K_i))\leq C(R_0) e^{-\zeta_it}, \quad \, \forall\,t\geq 0,\non
\ee
where $K_i>0$ and $\zeta_i>0$ $(i=2,3)$ are independent of $\beta$.
\ep
\begin{proof}
We briefly outline the proof for the case $i=2$. Applying $A_0$ to \eqref{e1c} and
testing the resultant by $-\Delta \phi^c_t(t)-\eta_3\Delta \phi^c(t)$, for some $\eta_3>0$, we get
 \be
\frac{d}{dt}\mathcal{Y}^c_2(t)+\mathcal{D}^c_2(t)\leq \mathcal{R}^c_2(t),\label{comx4a}
 \ee
 where
 \bea
 \mathcal{Y}_2^c&=&\frac{\beta}{2}\|\nabla \phi^c_t\|^2+\frac12
 \|\Delta^2 \phi^c\|^2-\|\nabla \Delta \phi^c\|^2+\frac{k}{2}\|\Delta \phi^c\|^2\non\\
 && +\eta_3\beta(\nabla \phi^c_t, \nabla \phi^c) +\frac{\eta_3}{2}\|\nabla \phi^c\|^2,\non\\
 \mathcal{D}_2^c&=& (1-\eta_3\beta)\|\nabla \phi^c_t\|^2 +\eta_3\|\Delta^2 \phi^c\|^2-2\eta_3\|\nabla \Delta \phi^c\|^2+k\eta_3\|\Delta \phi^c\|^2,\non\\
 \mathcal{R}_2^c&=&\int_Q\nabla \Delta(f(\phi)-f(\phi-\phi^c))\cdot \nabla\phi^c_t dx-k\int_Q \nabla \Delta \phi\cdot  \nabla \phi^c_t dx\non\\
 && -\eta_3\int_Q \Delta(f(\phi)-f(\phi-\phi^c))\Delta \phi^c dx+k\eta_3\int_Q \Delta \phi \Delta \phi^c dx.\non
 \eea
 Since $B_2 \in \mathbb{Y}_1^\beta$, we have
 \be
 \|(\phi^c(t), \phi_t^c(t))\|_{\mathbb{X}^\beta_1}\leq C(R_0), \quad  \|(\phi(t), \phi_t(t))\|_{\mathbb{X}^\beta_1}\leq C(R_0), \quad \forall\, t\geq 0,\label{comx1aa}
 \ee
 and the decay estimate \eqref{decayx0} still holds. Using the Sobolev embedding theorem, by the similar argument used
 in the previous lemma, we can choose $\eta_3>0$ and $\kappa_3>0$ independent of $\beta$ such that
 \be
 \frac{d}{dt}\mathcal{Y}^c_2(t)+\kappa_3\mathcal{Y}^c_2(t)\leq C(R_0),\label{comx5a}
 \ee
 and
 \be
 \mathcal{Y}^c_2(t)\geq C(\beta\|\nabla \phi^c_t\|^2+ \|\nabla \phi^c\|_{3}^2).\label{comx7a}
 \ee
 Since $\mathcal{Y}_2^c(0)=0$.
 Then, from \eqref{comx0}, \eqref{comx5a} and \eqref{comx7a} , we conclude that
 \be
 \|(\phi^c(t), \phi_t^c(t))\|_{\mathbb{X}_2^\beta}\leq C(R_0), \quad \forall\, t\geq 0.\label{comx0a}
 \ee
 Repeating the above argument, we can further get
 \be
 \|(\phi^c(t), \phi_t^c(t))\|_{\mathbb{X}_3^\beta}\leq C(R_0), \quad \forall\, t\geq 0,\label{comx0aa}
 \ee
 provided that $(\phi_0,\phi_1)\in B_3\subset \mathbb{Y}_2^\beta$.
 Together with the decay property \eqref{decayx0}, we can conclude the existence of compact attracting sets. The proof is complete.
\end{proof}

\section{Robust exponential attractor}
\label{robust}
\setcounter{equation}{0}
\noindent

In this section, we proceed to prove the main result Theorem \ref{main}, i.e., the existence of a family of exponential attractors for $\{S_\beta(t), \mathcal{X}_0^{M,M'}\}_{\beta\in [0,\beta_0]}$ that are, in particular, H\"older continuous with respect to the relaxation time $\beta$.

\subsection{Positively invariant attracting set in $\mathbb{Y}_3^\beta$}

First, we recall the transitivity property of exponential attraction (cf. \cite[Theorem 5.1]{FGMZ}):
\bl\label{trans}
Let $\mathbb{X}$ be a metric space with distance function denoted by ${\rm dist}$. $S(t)$ is a semigroup acting on $\mathbb{X}$ such that
$ {\rm dist}(S(t)z_1, S(t)z_2)\leq C_0e^{K_0t}{\rm dist}(z_1, z_2)$, for some $C, K>0$.  We further assume that there exist three subsets $B_1, B_2, B_3$  in $\mathbb{X}$ such that
$ {\rm dist}_{\mathbb{X}}(S(t)B_1,B_2)\leq C_1e^{-\alpha_1t}$, ${\rm dist}_{\mathbb{X}}(S(t)B_2,B_3)\leq C_2e^{-\alpha_2t}$.
Then we have
$${\rm dist}_{\mathbb{X}}(S(t)B_1,B_3)\leq C'e^{-\alpha't},\quad
\text{where} \ C'=C_0C_1+C_2\ \text{and}\ \alpha'=\frac{\alpha_1\alpha_2}{K_0+\alpha_1+\alpha_2}.$$
\el

Thus, we can prove the following

\bp \label{abs2}
For all $\beta\in [0, \beta_0]$, we have

(i) there exists a bounded closed set $\mathcal{B}_3$ in $\mathbb{Y}_3^\beta$ that exponentially attracts any bounded set of $\mathbb{Y}_0^\beta$ with respect to the $\mathbb{X}_0^\beta$-metric;

(ii) there exists a bounded positively invariant set $\mathcal{V}_3^\beta$ in
$\mathbb{Y}_3^\beta$, which absorbs the set $\mathcal{B}_3$ and, consequently, exponentially attracts
any bounded set of $\mathbb{Y}_0^\beta$ with respect to the $\mathbb{X}_0^\beta$-metric.
\ep
\begin{proof}
The conclusion (i) follows from Propositions \ref{atset1} and \ref{atset2} and Lemma \ref{trans}.
As far as (ii) is concerned, we infer from Proposition \ref{absh} the existence of a positively invariant and
$\mathbb{X}_3^\beta$-bounded set $\mathcal{V}_3^\beta$, which eventually absorbs any $\mathbb{X}_3^\beta$-bounded set of data. In particular, $\mathcal{V}_3^\beta$ absorbs $\mathcal{B}_3$, and by the definition of $\mathcal{B}_3$ in (i), we arrive at (ii). This ends the proof.
\end{proof}

\subsection{Smoothing property and Lipschitz continuity of $S_\beta(t)$}
Proposition \ref{abs2} enables us first to confine the dynamics on a regular positively invariant set $\mathcal{V}^\beta_3$ in $\mathbb{Y}_3^\beta$.
We note that it is not restrictive to assume $\mathcal{V}_3^\beta$ to be weakly closed in $\mathbb{X}_3^\beta$.
In what follows, we show the asymptotic smoothing property and Lipschitz continuity of the semigroup $S_\beta(t)$ on $\mathcal{V}_3^\beta$.

\begin{lemma}
\label{lmassm} Let $\beta\in (0, \beta_0]$. There exists $t^*\geq 0$ independent of $\beta$ such that, for the map $$\mathrm{S}_\beta:=S_\beta(t^*)$$
 we have
$$ \mathrm{S}_\beta\mathbf{u}_{01}-\mathrm{S}_\beta\mathbf{u}_{02}=D_\beta(\mathbf{u}_{01},\mathbf{u}_{02})
+K_\beta(\mathbf{u}_{01},\mathbf{u}_{02}),$$
for every $\mathbf{u}_{01}$, $\mathbf{u}_{02}\in \mathcal{V}_3^\beta$, where $D_\beta$ and $K_\beta$ satisfy
\be
\|D_\beta(\mathbf{u}_{01},\mathbf{u}_{02})\|_{\mathbb{X}_0^\beta}\leq\lambda\|\mathbf{u}_{01}-\mathbf{u}_{02}\|_{\mathbb{X}_0^\beta},\quad
\|K_\beta(\mathbf{u}_{01},\mathbf{u}_{02})\|_{\mathbb{X}_1^\beta}\leq \Lambda\|\mathbf{u}_{01}-\mathbf{u}_{02}\|_{\mathbb{X}_0^\beta},\label{asymsmoo}
\ee
for some $\lambda\in(0,\frac{1}{2})$ and $\Lambda\geq 0$ that are independent of $\beta$.

Besides, the map $(t,\mathbf{u})\mapsto S_\beta(t)\mathbf{u}: [t^*,2t^*]\times\mathcal{V}_3^\beta\rightarrow \mathcal{V}_3^\beta$
 is Lipschitz continuous when $\mathcal{V}_3^\beta$ is endowed with the $\mathbb{X}_0^\beta$-topology.
\end{lemma}
\begin{proof} We can argue as the proof for \cite[Lemma 5.3]{GW3} with minor modifications (mainly in order to stress the independence with respect to $\beta$). For the reader's convenience, we give a sketch of the proof.

For any $\mathbf{u}_{01}$, $\mathbf{u}_{02}\in \mathcal{V}_3^\beta$, we consider the weak solutions $(\phi_i, \phi_{it})(t)=S_\beta(t)\mathbf{u}_{0i}$ ($i=1,2$) to the MPFC equation \eqref{e1}--\eqref{e2} and we set
$$\mathbf{u}(t):=S_\beta(t)\mathbf{u}_{01}-S_\beta(t)\mathbf{u}_{02}=(\psi, \psi_t)(t),\quad  \mathbf{u}_0:=\mathbf{u}_{01}-\mathbf{u}_{02}=(\psi_0, \psi_1).$$
As in \cite{GW3}, we write the difference of solution $(\phi, \phi_t)$ as follows
$$  (\psi, \psi_t)(t)= (\psi^d, \psi^d_t)(t)+ (\psi^c, \psi^c_t)(t),$$
such that
 \bea
&& A_0^{-1}(\beta \psi^d_{tt}+\psi^d_t)+ \Delta^2\psi^d+ 2\Delta \psi^d+k \psi^d=0,\label{e1dd}
\\
&&\phi^d|_{t=0}=\overline{{\psi}_0}(x),\quad \phi^d_t|_{t=0}=\overline{\psi_1}(x), \label{e2dd}
\eea
 and
 \bea
&& A_0^{-1}(\beta \psi^c_{tt}+\psi^c_t)+\Delta^2 \psi^c+2\Delta \psi^c +f(\phi_1)-\langle f(\phi_1)\rangle\non\\
&&\qquad -f(\phi_2)+\langle f(\phi_2)\rangle = k(\psi-\psi^c),\label{e1cd}
\\
&&\psi^c|_{t=0}=\langle\psi_0(x)\rangle,\quad \psi^c_t|_{t=0}=\langle\psi_1(x)\rangle.\label{e2cd}
\eea
Here, $k>0$ is again a sufficiently large constant (not necessarily the same one used in the previous decomposition). For large $k$, it is easy to show the decay of $\psi^d$, which can be viewed as the solution to the linear problem \eqref{e1dd}--\eqref{e2dd}, namely,
 \be
\|\psi^d(t)\|_2^2 +\beta\|\psi_t^d(t))\|_{-1}^2\leq C\|(\overline{\psi_0}(x), \overline{\psi_1}(x))\|_{\mathbb{X}_0^\beta}^2e^{-\kappa t}, \quad \forall\, t\geq 0.\label{decayx0d}
 \ee
 Next, applying $A_0$ to \eqref{e1cd} and testing the resultant by $ \psi^c_t(t)+\eta_4 \psi^c(t)$, for some $\eta_4>0$, we get
 \be
\frac{d}{dt}\mathcal{Y}^c_3(t)+\mathcal{D}^c_3(t)\leq \mathcal{R}^c_3(t),\label{comx4d}
\ee
where
  \bea
\mathcal{Y}^c_3&=& \frac{\beta}{2}\|\psi^c_t\|^2+\frac12
 \|\nabla \Delta \psi^c\|^2-\|\Delta \psi^c\|^2+\frac{k}{2}\|\nabla \psi^c\|^2+\eta_4\beta(\psi^c_t, \psi^c)+\frac{\eta_4}{2}\|\psi^c\|^2,\non\\
 \mathcal{D}^c_3&=&(1-\eta_4\beta)\|\psi^c_t\|^2+\eta_4\|\nabla \Delta \psi^c\|^2-2\eta_4\|\Delta \psi^c\|^2+\eta_4 k\|\nabla \psi^c\|^2,\non\\
 \mathcal{R}^c_3&=&\int_Q\Delta(f(\phi_1) -f(\phi_2))\psi^c_t dx-k\int_Q \Delta \psi \psi^c_t dx\non\\
 && \quad +\eta_4 \int_Q(f(\phi_1)-f(\phi_2))\Delta \psi^c dx-\eta_4 k\int_Q \psi \Delta \psi^c dx\non
\eea
The argument used to get \eqref{comx7} easily yields, for sufficiently large $k$ and small $\eta_4$ that may depend on $\beta_0$ but not on $\beta$, that
 \be
 \mathcal{D}^c_3(t)\geq \frac12\|\psi_t^c\|^2+\frac{\eta_4k}{2} \|\nabla \psi^c\|^2+\kappa_4\mathcal{Y}^c_3(t)\geq C(\beta\|\psi^c_t\|^2+ \| \psi^c\|_{3}^2),\label{comx7d}
 \ee
 for suitably small $\kappa_4>0$.

 Therefore,  using the uniform $\mathbb{X}_0^\beta$-estimates of $(\phi_i, \phi_{it})$ and $(\psi^c, \psi^c_t)$,   the remainder term  $\mathcal{R}^c_3$ can be estimated by
 \bea
 \mathcal{R}_3^c(t)
 &\leq& \frac12\|\psi_t^c\|^2+\frac{\eta_4k}{2} \|\nabla \psi^c\|^2+ C(\|\phi_1\|_{2}, \|\phi_2\|_{2})\|\psi\|_2^2,\non
 \eea
 which implies
 \be
\frac{d}{dt}\mathcal{Y}^c_3(t)+\kappa_4\mathcal{Y}^c_3(t)\leq C(\|\phi_1\|_{2}, \|\phi_2\|_{2})\|\psi\|_2^2.\label{comx5d}
\ee
Integrating \eqref{comx5d} with respect to time, we infer from the choice of initial data and the Lipschitz continuity estimate \eqref{LipX0} that
 \bea
 \mathcal{Y}^c_3(t) &\leq& \mathcal{Y}^c_3(0)+ \int_0^t C(\|\phi_1(s)\|_{2}, \|\phi_2(s)\|_{2})\|\psi(s)\|_2^2ds
 \non\\
 &\leq&  C(t)\|(\psi_0, \psi_1)\|_{\mathbb{X}_0^\beta}^2.\label{KK}
 \eea
 Due to \eqref{decayx0d}, for any fixed $\lambda\in (0,\frac12)$,
 we can choose $t^*$ sufficiently large such
 that
 \be
 \|(\psi^d(t^*), \psi_t^d(t^*))\|_{\mathbb{X}_0^\beta}\leq \lambda \|(\psi_0(x), \psi_1(x))\|_{\mathbb{X}_0^\beta}.\label{DD}
 \ee
 Fix such $t^*$ and set
 \be
 {\rm S}_\beta=S_\beta(t^*), \quad D_\beta(\mathbf{u}_{01},\mathbf{u}_{02})=(\psi^d(t^*), \psi_t^d(t^*)),\quad
 K_\beta(\mathbf{u}_{01},\mathbf{u}_{02})=(\psi^c(t^*),\psi^c_t(t^*)).\non
 \ee
 It follows from \eqref{KK} and \eqref{DD} that \eqref{asymsmoo} holds.

Next, for any $t,\tau\in [t^*, 2t^*]$ satisfying $t\geq \tau$ and $\mathbf{u}_1, \mathbf{u}_2\in \mathcal{V}_3^\beta$,
we have
\bea
&& \|S_\beta(t)\mathbf{u}_{01}-S_\beta(\tau)\mathbf{u}_{02}\|_{\mathbb{X}_0^\beta}^2\non\\
&\leq& 2 \|S_\beta(t) \mathbf{u}_{01}-S_\beta(t)\mathbf{u}_{02}\|_{\mathbb{X}_0^\beta}^2+ 2 \|S_\beta(t) \mathbf{u}_{02}-S_\beta(\tau)\mathbf{u}_{02}\|_{\mathbb{X}_0^\beta}^2,\label{holder}
\eea
where the first term on the right-hand side can be easily estimated like in \eqref{LipX0}.
Recalling that the initial datum is in $\mathcal{V}_3^\beta$, we have the uniform estimate
$$\|S_\beta(t) \mathbf{u}_{02}\|_{\mathbb{X}_3^\beta}\leq C\big(\|\mathbf{u}_{02}\|_{\mathbb{X}_3^\beta}\big),$$
which together with equation \eqref{e1} also implies  $\|\phi_{2tt}(t)\|_{-1}\leq C_\beta$ ($C_\beta$ depends on $\beta$).
Then for the second term on the right-hand side of \eqref{holder}, we infer that
\bea
\|S_\beta(t) \mathbf{u}_{02}-S_\beta(\tau)\mathbf{u}_{02}\|_{\mathbb{X}_0^\beta}^2
&=& \|\phi_2(t)-\phi_2(\tau)\|_2^2+\beta \|\phi_{2t}(t)-\phi_{2t}(\tau)\|_{-1}^2\non\\
&\leq& \left(\int_\tau^t \|\phi_{2t}(s)\|_2 ds\right)^2+ \beta \left(\int_\tau^t \|\phi_{2tt}(s)\|_{-1} ds\right)^2\non\\
&\leq&  C_\beta|t-\tau|^2.\non
\eea
As a consequence, we deduce the Lipschitz continuity
\be
\|S_\beta(t)\mathbf{u}_{01}-S_\beta(\tau)\mathbf{u}_{02}\|_{\mathbb{X}_0^\beta}\leq C(\beta, t^*)\Big(\|\mathbf{u}_{01}-\mathbf{u}_{02}\|_{\mathbb{X}_0^\beta}+ |t-\tau|\Big),\non
\ee
where $C(\beta, t^*)$ is a constant depending on $\beta$, $t^*$ and $\mathbb{X}_0^\beta$-norm of the initial data. This concludes the proof.
\end{proof}

%

\subsection{Rescaled operator and boundary layer estimate}

In the spirit of \cite{MPZ}, we now introduce a suitable rescaling of the semigroup $S_\beta(t)$.
More precisely, for $\beta\in (0, \beta_0]$, we define
\be
\mathcal{T}_\beta(u,v)=\Big(u, \sqrt{\beta\beta_0^{-1}} v\Big): \mathbb{Y}_i^{\beta}\to \mathbb{Y}_i^{\beta_0}, \quad i=0,1,2,3.\non
\ee
For all $\mathbf{u}=(u,v)\in \mathbb{Y}_i^{\beta}$ ($i=0,1,2,3$), we have
\be
\|\mathcal{T}_\beta \mathbf{u}\|_{\mathbb{X}_i^{\beta_0}}=\|\mathbf{u}\|_{\mathbb{X}_i^\beta}.\non
\ee
Then the rescaled semigroup $\widehat{S}_\beta(t): \mathbb{Y}_i^{\beta_0}\to \mathbb{Y}_i^{\beta_0}$ is given by
\be
\widehat{S}_\beta(t)(u,v)=
\begin{cases}
&\mathcal{T}_\beta S_\beta(t)\mathcal{T}_\beta^{-1}(u,v), \quad \text{for} \ \beta\in (0,\beta_0],\\
& S_0(t)(u,0),\qquad \qquad \text{for}\ \beta=0.
\end{cases}
\ee
\par For any $\beta\in (0, \beta_0]$ and initial data $(\phi_0, \phi_1)\in \mathcal{V}_3^{\beta}$, it follows from Proposition \ref{absh} that
the corresponding solution $(\phi, \phi_t)=S_\beta(t)(\phi_0, \phi_1)$ satisfies the uniform estimate $$ \|\phi(t)\|_5\leq C, \quad \forall\, t\geq 0$$ and the bound is independent of $\beta$.
 Denote $v=\phi_t$, we can view \eqref{e1} as
$$ \beta v_t+v= G:=\Delta[\Delta^2 \phi+2\Delta \phi+f(\phi)]\in H^{-1}_p, \quad \text{with}\ v|_{t=0}=\phi_1.$$
Solving the above equation, it follows that $$v(t)=v(0)e^{-\frac{1}{\beta}t}+ \frac{1}{\beta}e^{-\frac{1}{\beta}t}\int_0^t e^{\frac{1}{\beta}s}G(s) ds$$
and
\bea
\|v(t)\|_{-1}&\leq& \|v(0)\|_{-1}e^{-\frac{1}{\beta}t}+ \sup_{s\in[0,t]} \|G(s)\|_{-1}\non\\
&\leq& \Big(\beta^{-\frac12}\|(\phi_0, \phi_1)\|_{\mathbb{X}_0^\beta}\Big)e^{-\frac{1}{\beta}t}+C\Big(\sup_{s\in[0,t]}\|\phi(s)\|_5\Big).\label{phita}
\eea
Using the simple fact $\lim_{\beta\to 0+}\beta^{-\frac12}e^{-\frac{1}{\beta}}=0$, we see that
 \be \|\phi_t(t)\|_{-1}\leq C, \quad \forall\, t\geq 1,\label{aphites}
 \ee where the bound is uniform for $\beta\in (0, \beta_0]$.

 In summary, we have for all $\mathbf{u}_0=(\phi_0, \phi_1)\in \mathcal{V}_3^{\beta}$, it holds
\be
\|S_\beta(t) \mathbf{u}_0\|_{\mathbb{X}_0^{\beta_0}}\leq C\big(\|\mathbf{u}_0\|_{\mathbb{X}_3^\beta}, \beta_0\big),\quad \forall \, t\geq 1,\label{dif1}
\ee
 where the bound may depend on $\beta_0$ but is independent of $\beta$.

\subsection{H\"{o}lder continuity with respect to $\beta$}

\bl\label{ho}
For any $0\leq \beta_2< \beta_1 \leq \beta_0$, $\mathbf{u}_0\in \mathcal{V}_3^{\beta_0}$, there holds
\be
\|\widehat{S}_{\beta_1}(t)\mathbf{u}_0-\widehat{S}_{\beta_2}(t)\mathbf{u}_0\|_{\mathbb{X}_0^{\beta_0}}\leq K_1 e^{K_2 t}(\beta_1-\beta_2)^\frac16, \quad \forall\, t\geq 1,\label{ddes}
\ee
where the constants $K_1, K_2>0$ may depend on $\beta_0$ but are independent of $\beta_1$ and $\beta_2$.
\el
\begin{proof}
 \textbf{Case 1}. We first consider the case $\beta_2=0$. For any $\mathbf{u}_0=(\phi_0, \phi_1)\in \mathcal{V}_3^{\beta_0}$, we denote by $$\phi^0\in L^\infty(0, T; H^5_p(Q))\cap H^1(0,T; H^2_p(Q))$$ the solution to the PFC equation \eqref{pfc}--\eqref{pfc2} with initial datum $\phi_0$ and by $$(\phi^{\beta_1}, \phi_t^{\beta_1})\in L^\infty(0,T; \mathbb{X}_3^{\beta_1})\cap W^{1,\infty}(0, T; \mathbb{X}_0^{\beta_1})$$ the solution to the MPFC equation \eqref{e1}--\eqref{e2} with initial data $(\phi_0, \sqrt{\beta_0{\beta_1}^{-1}} \phi_1)\in \mathcal{V}_3^{\beta_1}$.

Then the difference $\psi=\phi^0 -\phi^{\beta_1}$ satisfy
\be
\psi_t-\Delta(\Delta^2 \psi+2\Delta \psi)=\Delta (f(\phi^0)-f(\phi^{\beta_1}))+ {\beta_1} \phi^{\beta_1} _{tt},\quad \text{in} \ Q,\label{e1aa}
 \ee
 with initial data $\psi|_{t=0}=0$. Testing \eqref{e1aa} by $A_0^{-1}\overline{\psi_t}(t)\in H^4_p(Q)$, we get
\bea
&& \frac12 \frac{d}{dt} \Big(\|\Delta \psi\|^2-2\|\nabla \psi\|^2+2\|\overline{\psi}\|^2 \Big)+\|\overline{\psi_t}\|_{-1}^2\non\\
&=& -\int_Q (f_2(\phi^0)-f_2(\phi^{\beta_1})-\langle f_2(\phi^0)\rangle+\langle f_2(\phi^{\beta_1})\rangle \overline{\psi_t} dx+ \beta (A_0^{-\frac12} \phi^{\beta_1}_{tt}, A_0^{-\frac12} \overline{\psi_t}),\non
\eea
where $f_2(s):=f(s)+2s$. On account of \eqref{intpo}, we see that
\be
\mathcal{G}(\psi(t))=\|\Delta \psi\|^2-2\|\nabla \psi\|^2+2\|\overline{\psi}\|^2 \geq c_Q\|\overline \psi\|_2^2.\label{GG}
\ee
Using Proposition \ref{absh} and the Sobolev embedding theorem $H^2(Q)\hookrightarrow L^\infty(Q)$ ($n\leq 3$), we have
\bea
 \frac12 \frac{d}{dt} \mathcal{G}(\psi)+\|\overline{\psi_t}\|_{-1}^2
&\leq & C(\|\phi^0\|_2, \|\phi^{\beta_1}\|_2)\|\psi\|_1 \| \overline{\psi_t}\|_{-1}+ {\beta_1} \| \phi^{\beta_1}_{tt}\|_{-1}\| \overline{\psi_t}\|_{-1}\non\\
&\leq& \frac12 \| \overline{\psi_t}\|_{-1}^2+ C(\|\overline{\psi}\|_2^2+|\langle\psi\rangle|^2)+{\beta_1}^2 \| \phi^{\beta_1}_{tt}\|_{-1}^2\non\\
&\leq& \frac12 \| \overline{\psi_t}\|_{-1}^2+ C\mathcal{G}(t)+C|\langle\psi\rangle|^2+{\beta_1}^2 \| \phi^{\beta_1}_{tt}\|_{-1}^2.\non
\eea
Since $\psi(0)=0$, then applying the Gronwall lemma with \eqref{esphitt} (see Corollary \ref{phitt}) and the mass conservation property for $\phi^0$ and $\phi^{\beta_1}$ (cf. \eqref{conpfc} and \eqref{mde2}), we obtain
\bea
 \|\psi(t)\|_2^2 &\leq & C(\|\overline{\psi}(t)\|_2^2+|\langle\psi\rangle|^2)\non\\
 & \leq & C(\mathcal{G}(\psi(t))+|\langle\psi\rangle|^2)\non\\
 &\leq&  Ce^{Ct}{\beta_1}  \int_0^t e^{-Cs} \big(\beta_1\| \phi^{\beta_1}_{tt}\|_{-1}^2+\beta_1  |\langle \phi_1\rangle|^2\big) ds + C{\beta_1}^2|\langle \phi_1\rangle|^2\non\\
 &\leq&  Ce^{Ct} t {\beta_1} + C{\beta_1}^2.\label{ddp1}
\eea
Besides, we infer from \eqref{aphites}  that
\be
\|\phi^{\beta_1}_t(t)-0\|_{-1}\leq C, \quad \forall \,t\geq 1,\label{ddp2}
\ee
where $C$ is uniform for  $\beta\in (0, \beta_0]$. As a result, it follows from \eqref{ddp1} and \eqref{ddp2} that
\bea
\|\widehat{S}_{\beta_1}(t)\mathbf{u}_0-\widehat{S}_0(t)\mathbf{u}_0\|_{\mathbb{X}^{\beta_0}_0}^2&=&\|(\phi^{\beta_1}(t), \phi^{\beta_1}_t(t))-(\phi^0(t), 0)\|_{\mathbb{X}^{\beta_1}_0}^2\non\\
&\leq & Ct e^{Ct}({\beta_1}^2+\beta_1)\non\\
&\leq & Ct e^{Ct}\beta_1, \quad \forall\, t\geq 1,\label{dda}
\eea
where $C$ may depend on $\beta_0$.

\textbf{Case 2}. We consider the case $0<\beta_2<\beta_1\leq \beta_0$. As before, we denote by $$(\phi^{\beta_i}, \phi_t^{\beta_i})\in L^\infty(0,T; \mathbb{X}_3^{\beta_i})\cap W^{1,\infty}(0, T; \mathbb{X}_0^{\beta_i})$$
the solutions to MPFC equation \eqref{e1}--\eqref{e2} with initial data such that $(\phi_0, \sqrt{\beta_0{\beta_i}^{-1}} \phi_1)\in \mathcal{V}_3^{\beta_i}$, for $i=1,2$.
Then the difference $(\psi, \psi_t)=(\phi^{\beta_1}-\phi^{\beta_2}, \phi^{\beta_1}_t-\phi^{\beta_2}_t)$
satisfies
\bea
&&\beta_2 \psi_{tt}+\psi_t-\Delta[\Delta^2 \psi+2\Delta \psi]=\Delta (f(\phi^{\beta_1})-f(\phi^{\beta_2}))+ (\beta_2-\beta_1) \phi^{\beta_1} _{tt},\label{e1ab}\\
&& \psi|_{t=0}=0, \quad \psi_t|_{t=0}=\Big(\sqrt{\beta_0{\beta_1}^{-1}}-\sqrt{\beta_0{\beta_2}^{-1}}\Big)\phi_1.
 \eea
 Testing \eqref{e1ab} by $A_0^{-1}\overline{\psi}(t)$, we obtain
 \bea
&& \frac{d}{dt}\left(\frac{\beta_2}{2}\|\overline{\psi}_t\|_{-1}^2+\frac12\|\Delta \psi\|^2-\|\nabla \psi\|^2+\|\overline{\psi}\|^2\right)+\|\overline{\psi_t}\|_{-1}^2\non\\
 &=& -\int_Q (f_2(\phi^{\beta_1})-f_2(\phi^{\beta_2})-\langle f_2(\phi^{\beta_1})\rangle+\langle f_2(\phi^{\beta_2})\rangle) \overline{\psi_t} dx\non\\
 && + (\beta_2-\beta_1)(A_0^{-\frac12} \phi^{\beta_1}_{tt}, A_0^{-\frac12} \overline{\psi_t}).\non
\eea
Using Proposition \ref{absh} and the Sobolev embedding theorem, we have
\bea
&& \frac{d}{dt}\left(\frac{\beta_2}{2}\|\overline{\psi}_t\|_{-1}^2+\mathcal{G}(\psi)\right)+\|\overline{\psi_t}\|_{-1}^2\non\\
&\leq & C(\|\phi^{\beta_1}\|_2, \|\phi^{\beta_2}\|_2)\|\psi\|_1 \| \overline{\psi_t}\|_{-1}+ |\beta_1-\beta_2| \| \phi^{\beta_1}_{tt}\|_{-1}\| \overline{\psi_t}\|_{-1}\non\\
&\leq& \frac12 \| \overline{\psi_t}\|_{-1}^2+ C(\|\overline{\psi}\|_2^2+|\langle\psi\rangle|^2)+|\beta_1-\beta_2|^2\| \phi^{\beta_1}_{tt}\|_{-1}^2\non\\
&\leq& \frac12 \| \overline{\psi_t}\|_{-1}^2 +C\mathcal{G}(\psi(t))+C|\langle\psi\rangle|^2+\frac{|\beta_1-\beta_2|^2}{\beta_1}(\beta_1 \| \phi^{\beta_1}_{tt}\|_{-1}^2),\non
\eea
where $\mathcal{G}(\psi)$ is as in \eqref{GG}.

Using the fact $0<\beta_2<\beta_1\leq \beta_0$ and the mass conservation property \eqref{mde2}, we can calculate
that, for $t\geq 1$,
\bea
\beta_2 |\langle \psi_t\rangle|^2 &=& \beta_0 |\langle \phi_1\rangle|^2\left(\sqrt{\beta_2{\beta_1}^{-1}} e^{-\frac{t}{\beta_1}}-e^{-\frac{t}{\beta_2}}\right)^2\non\\
&\leq&  C\frac{\beta_2}{\beta_1}\left(e^{-\frac{t}{\beta_1}}-e^{-\frac{t}{\beta_2}}\right)^2+C \left(\sqrt{\beta_2{\beta_1}^{-1}} -1\right) ^2 e^{-\frac{2t}{\beta_2}}\non\\
&\leq & 2C\frac{\beta_2}{\beta_1}\left|e^{-\frac{t}{\beta_1}}-e^{-\frac{t}{\beta_2}}\right|
+C\frac{(\sqrt{\beta_1}-\sqrt{\beta_2})^2}{\beta_1}\non\\
& \leq & Ct \frac{\beta_2}{\beta_1} \left|\frac{1}{\beta_1}-\frac{1}{\beta_2}\right|+ C\frac{(\sqrt{\beta_1}-\sqrt{\beta_2})(\sqrt{\beta_1}+\sqrt{\beta_2})}{\beta_1}\non\\
&\leq& Ct\frac{\beta_1-\beta_2}{\beta_1^2}  + C\frac{\beta_1-\beta_2}{\beta_1}\non\\
&\leq& Ct\frac{\beta_1-\beta_2}{\beta_1^2},\non
\eea
and, by a similar argument, we deduce
\bea
|\langle \psi\rangle|^2&=& \beta_0|\langle \phi_1\rangle|^2 \left( \sqrt{\beta_1}-\sqrt{\beta_1}e^{-\frac{t}{\beta_1}}-\sqrt{\beta_2}+\sqrt{\beta_2}e^{-\frac{t}{\beta_2}}\right)^2\non\\
&\leq& C\left( \sqrt{\beta_1}-\sqrt{\beta_2}\right)^2+ C \beta_2 \left(e^{-\frac{t}{\beta_1}}-e^{-\frac{t}{\beta_2}}\right)^2\non\\
&\leq& C(\beta_1-\beta_2)+  Ct\frac{\beta_1-\beta_2}{\beta_1}\non\\
&\leq& Ct\frac{\beta_1-\beta_2}{\beta_1}.\non
\eea
We note that in the above two estimates the constants $C$ may depend on $\beta_0$ but are independent of $\beta_1, \beta_2$.

Then by the Gronwall lemma and \eqref{esphitt}, we obtain that (always keeping in mind that $\beta_2<\beta_1\leq \beta_0$)
\bea
&&\|\psi(t)\|^2_2+\beta_2\|\psi_t(t)\|_{-1}^2 \non\\
&\leq& C(\mathcal{G}(\psi(t))+ |\langle\psi(t)\rangle|^2)+\beta_2\|\overline{\psi_t}(t)\|_{-1}^2+\beta_2 |\langle \psi_t\rangle|^2\non\\
&\leq &  Ce^{Ct} \beta_2\|\overline{\psi_t}(0)\|_{-1}^2+ Ce^{Ct}\int_0^t  e^{-Cs} \left[ \frac{|\beta_1-\beta_2|^2}{\beta_1}\big(\beta_1 \| \phi^{\beta_1}_{tt}\|_{-1}^2)+C |\langle\psi\rangle|^2\right]ds\non\\
&& + C|\langle\psi(t)\rangle|^2+\beta_2 |\langle \psi_t\rangle|^2 \non\\
&\leq& C e^{Ct} \beta_0\Big(\sqrt{\beta_2{\beta_1}^{-1}}-1\Big)^2\|\phi_1\|_{-1}^2+ \frac{|\beta_1-\beta_2|^2}{\beta_1}Ce^{Ct}t\non\\
&& +C\left(e^{Ct}+t\right) \frac{\beta_1-\beta_2}{\beta_1}+Ct\frac{\beta_1-\beta_2}{\beta_1^2}\non\\
&\leq& Cte^{Ct} \frac{\beta_1-\beta_2}{\beta_1^2},\quad \forall\, t\geq 1,\label{ddp3}
\eea
 where $C$ may depend on $\beta_0$ but is independent of $\beta_1, \beta_2$. As a result, it follows from \eqref{aphites} and \eqref{ddp3} that
\bea
&&\|\widehat{S}_{\beta_1}(t)\mathbf{u}_0-\widehat{S}_{\beta_2}(t)\mathbf{u}_0\|_{\mathbb{X}^{\beta_0}_0}^2\non\\
&=& \|\phi^{\beta_1}(t)-\phi^{\beta_2}(t)\|_2^2+ \|\beta_1^\frac12 \phi^{\beta_1}_t(t)-\beta_2^\frac12\phi_t^{\beta_2}(t))\|_{-1}^2\non\\
&\leq& \|\psi(t)\|_2^2+ 2\beta_2\|\psi_t(t)\|_{-1}^2+2|\beta_1^\frac12-\beta_2^\frac12|^2\|\phi^{\beta_1}_t(t)\|_{-1}^2\non\\
&\leq& Ct e^{Ct} \frac{\beta_1-\beta_2}{\beta_1^2}+C\frac{(\beta_1-\beta_2)^2}{(\beta_1^\frac12+\beta_2^\frac12)^2}\non\\
&\leq& Ct e^{Ct} \frac{\beta_1-\beta_2}{\beta_1^2}, \quad \forall\, t\geq 1.\label{ddb}
\eea
On the other hand, \eqref{dda} yields that
\bea
&&\|\widehat{S}_{\beta_1}(t)\mathbf{u}_0-\widehat{S}_{\beta_2}(t)\mathbf{u}_0\|_{\mathbb{X}^{\beta_0}_0}^2\non\\
&\leq&
2\|\widehat{S}_{\beta_1}(t)\mathbf{u}_0-\widehat{S}_0(t)\mathbf{u}_0\|_{\mathbb{X}^{\beta_0}_0}^2
+2\|\widehat{S}_{\beta_2}(t)\mathbf{u}_0-\widehat{S}_{0}(t)\mathbf{u}_0\|_{\mathbb{X}^{\beta_0}_0}^2\non\\
&\leq& Ct e^{Ct}\beta_1, \quad \forall\, t\geq 1.\label{ddc}
\eea
Using the following elementary inequality (valid for $0\leq \beta_2<\beta_1\leq \beta_0$)
\be
\min\left\{\beta_1, \frac{\beta_1-\beta_2}{\beta_1^2}\right\}\leq (\beta_1-\beta_2)^\frac13, \non
\ee
we infer from \eqref{ddb} and \eqref{ddc} that
\be
\|\widehat{S}_{\beta_1}(t)\mathbf{u}_0-\widehat{S}_{\beta_2}(t)\mathbf{u}_0\|_{\mathbb{X}^{\beta_0}_0}\leq  C\sqrt{t} e^{Ct}(\beta_1-\beta_2)^\frac16, \quad \forall\, t\geq 1.\label{ddd}
\ee
Combining \eqref{dda} and \eqref{ddd} and recalling that $\beta_1\leq \beta_0$, we arrive at our conclusion, that is, estimate \eqref{ddes}. The proof is complete.
\end{proof}

\subsection{Proof of Theorem \ref{main}}

After the preparations in previous sections, we are now able to employ the argument devised in \cite{MPZ} for the damped wave equation. For the reader's convenience, we briefly summarize the proof below.

\textbf{Step 1}. Let $t^*\geq 0$ be the one determined in Lemma \ref{lmassm} (independent of $\beta$) and the map given by $\mathrm{S}_\beta=S_\beta(t^*)$. Consider the rescaled discrete map
\be
\widehat{\mathrm{S}}_\beta(u,v)=
\begin{cases}
&\mathcal{T}_\beta \mathrm{S}_\beta\mathcal{T}_\beta^{-1}(u,v), \quad \text{for} \ \beta\in (0,\beta_0],\\
& \mathrm{S}_0(u,0),\qquad \qquad \text{for}\ \beta=0.
\end{cases}
\ee
Using Lemma \ref{lmassm}, Lemma \ref{ho} and the boundary layer estimate \eqref{dif1}, we can apply the abstract result in \cite{EMZ} (see also \cite{EMZ05}) to the rescaled discrete semigroup $(\widehat{\mathrm{S}}_\beta)^m$ ($m\in \mathbb{N}$) such that there exists a family of compact sets
$\mathcal{M}^{\rm d}_\beta\subset \mathcal{V}_3^\beta$ positively invariant under $\mathrm{S}_\beta$ and uniformly bounded in $\mathbb{X}_0^{\beta_0}$ such that
\bea
&&{\rm dist}_{\mathbb{X}_0^\beta}((\mathrm{S}_\beta)^m\mathcal{V}_3^\beta,\mathcal{M}^{\rm d}_\beta)\leq Ce^{-\gamma m}\quad \text{and}\quad {\rm dim}_{\mathbb{X}_0^\beta}\mathcal{M}^{\rm d}_\beta\leq C,\\
&& {\rm dist}_{\mathbb{X}_0^{\beta_1}}^{{\rm sym}}(\mathcal{M}^{\rm d}_{\beta_1}, \mathcal{M}^{\rm d}_{\beta_2})\leq C(\beta_1-\beta_2)^\frac16, \quad \text{for}\ 0\leq \beta_2<\beta_1\leq \beta_0.\label{Hodis}
\eea

\textbf{Step 2}. Set $$\mathcal{M}_\beta=\bigcup_{t\in [t^*, 2t^*]}S_\beta(t)\mathcal{M}^{\rm d}_\beta.$$ The positive invariance of $\mathcal{M}^{\rm d}_\beta$ indicates that $\mathcal{M}_\beta$ is also positively invariant. Property (P1) follows from Lemma \ref{absh} and \eqref{dif1}. Proposition \ref{abs2} yields the uniform exponential attraction property
\be
{\rm dist}_{\mathbb{X}_0^\beta}(S(t)\mathcal{V}^\beta_3, \mathcal{M}_\beta)\leq Ke^{-\gamma t},\quad \forall\, t\geq 0.\label{exp1}
\ee
(P3) follows from the finite dimensionality of $\mathcal{M}^{\rm d}_{\beta}$ and the Lipschitz continuity of map $(t, \mathbf{u}_0)\to S_\beta(t)\mathbf{u}_0$ from $[t^*, 2t^*]\times \mathcal{M}^{\rm d}_{\beta}$ to $\mathcal{M}^{\rm d}_{\beta}$ given by Lemma \ref{lmassm}. At last, the H\"older continuity \eqref{Hodis} for the discrete exponential attractor $\mathcal{M}^{\rm d}_{\beta}$ together with the continuous dependence \eqref{LipX0} (for $S_\beta(t)$) and \eqref{LipX00} (for $S_0(t)$) yields property (P4).

\textbf{Step 3}. Finally, we prove (P2), i.e., the basin of exponential attraction coincides with $\mathcal{X}_0^{M,M'}$ (recall \eqref{exp}) instead of the much more regular set $\mathcal{V}_3^\beta$. This is a direct consequence of the uniform exponential attraction property \eqref{exp1}, the Lipschitz continuity \eqref{LipX0} and the transitivity of the exponential attraction (cf. Lemma \ref{trans}).

Thus, the proof of Theorem \ref{main} is now complete. \quad \quad $\square$

%

\section*{Acknowledgments} M. Grasselli gratefully acknowledges the support of Shanghai Key Laboratory for Contemporary Mathematics of Fudan University through a Senior Visiting Scholarship 2013--2014. H. Wu was
partially supported by National Science Foundation of China 11371098, SRFDP and ``Chen Guang" project supported by Shanghai Municipal Education Commission and Shanghai Education Development Foundation.


\begin{thebibliography}{10}
\itemsep=0pt


\bibitem{BRV} R. Backofen, A. R\"{a}tz and A. Voigt, Nucleation and growth by a phase field crystal (PFC)
model, Phil. Mag. Lett., \textbf{87} (2007), 813--820.

\bibitem{BHLWWZ} A. Baskaran, Z. Hu, J.-S. Lowengrub, C. Wang, S. Wise and P. Zhou, Stable and efficient finite-difference nonlinear-multigrid schemes for the modified phase-field crystal equation, J. Comput. Phys., \textbf{228} (2009), 5323--5339.

\bibitem{BLWW} A. Baskaran, J.S. Lowengrub, C. Wang and S.M. Wise, Convergence analysis of a second order convex splitting scheme for the modified phase field crystal equation, SIAM J. Numer. Anal., \textbf{51}  (2013),  2851--2873.


\bibitem{CW} M. Cheng and J.A. Warren, An efficient algorithm for solving the phase field crystal model,
J. Comput. Phys.,  \textbf{227}  (2008),  6241--6248.

\bibitem{CKLE} M. Cheng, J. Kundin, D. Li and H. Emmerich, Thermodynamic consistency and
fast dynamics in phase-field crystal modeling, Phil. Mag. Lett., \textbf{92} (2012), 517--526.

\bibitem{EKHG} K.-R. Elder, M. Katakowski, M. Haataja and M. Grant, Modeling elasticity in crystal
growth, Phys. Rev. Lett., \textbf{88} (2002), 245701.

\bibitem{EG} K.-R. Elder and M. Grant, Modeling elastic and plastic deformations in nonequilibrium processing
using phase-field crystal, Phys. Rev. E, \textbf{90} (2004), 051605.

\bibitem{EMZ} M. Efendiev, A. Miranville and S. Zelik, Exponential attractors for a nonlinear reaction-diffusion system
in $\mathbb{R}\sp 3$,  C. R. Acad. Sci. Paris S\'er. I Math., \textbf{330} (2000), 713--718.

\bibitem{EMZ05} M. Efendiev, A. Miranville and S. Zelik, Exponential attractors and finite-dimensional reduction for nonautonomous dynamical systems, Proc. Roy. Soc. Edinburgh Sect. A, \textbf{13} (2005), 703--730.

\bibitem{EGL11} H. Emmerich, L. Gr\'{a}n\'{a}sy and H. L\"{o}wen, Selected issues of phase-field crystal simulations,
Eur. Phys. J. Plus, \textbf{126} (2011), 102.

\bibitem{ELWGTTG12} H. Emmerich, H. L\"{o}wen, R. Wittkowskib, T. Gruhna,
G.I. T\'{o}th, G. Tegze and L. Gr\'{a}n\'{a}sy, Phase-Field-crystal models for condensed matter dynamics on
atomic length and diffusive time scales: an overview,  Adv. Phys., \textbf{61} (2012), 665--743.

\bibitem{FGMZ} P. Fabrie, C. Galusinski, A. Miranville and S. Zelik, Uniform exponential attractors for a singularly
perturbed damped wave equation, Discrete Contin. Dyn. Syst., \textbf{10} (2004), 211--238.

\bibitem{GJ05} P. Galenko and D. Jou, Diffuse-interface model for rapid phase transformations in
nonequilibrium systems, Phys. Rev. E, \textbf{71} (2005), 046125.

\bibitem{GDL09} P. Galenko, D. Danilov and V. Lebedev, Phase-field-crystal and Swift--Hohenberg equations with fast dynamics, Phys. Rev. E, \textbf{79} (2009), 051110.

\bibitem{GE11} P. Galenko and K. Elder, Marginal stability analysis of the phase field crystal model in one spatial dimension, Phys. Rev. B, \textbf{83} (2011), 064113.

\bibitem{GGKE} P. Galenko, H. Gomez, N.V. Kropotin and K.R. Elder, Unconditionally stable method and numerical solution of the hyperbolic (modified)
phase-field crystal equation, Phys. Rev. E, to appear.


\bibitem{GN} H. Gomez and X. Nogueira, An unconditionally energy-stable method for the phase field crystal equation,
Comput. Methods Appl. Mech. Engrg.,  \textbf{249/252} (2012), 52--61.




\bibitem{GW3} M. Grasselli and H. Wu, Well-posedness and long-time behavior for the modified
phase-field crystal equation, preprint, 2013.

\bibitem{HWWL09} Z. Hu, S. Wise, C. Wang and J. Lowengrub, Stable and efficient finite-difference nonlinear-multigrid schemes for the phase-field crystal equation, J. Comput. Phys., \textbf{228} (2009), 5323--5339.

\bibitem{MPZ} A. Miranville, V. Pata and S. Zelik, Exponential attractors for singularly
perturbed damped wave equations: a simple construction, Asymptot. Anal. \textbf{53} (2007), 1--12.

\bibitem{MZ} A. Miranville and S. Zelik, Attractors for dissipative partial differential equations in bounded and unbounded domains. Handbook of
    differential equations: evolutionary equations. Vol. IV, 103--200, Handb. Differ. Equ., Elsevier/North-Holland, Amsterdam, 2008.

\bibitem{OS08} H. Ohnogi and Y. Shiwa, Instability of spatially periodic patterns due to a zero mode in the phase-field crystal equation, Phys. D, \textbf{237} (2008), 3046--3052.

\bibitem{PDA} N. Provatas, J. A. Dantzig, B. Athreya, P. Chan, P. Stefanovic, N. Goldenfeld and
K.R. Elder, Using the phase-field crystal method in the multi-scale modeling of microstructure evolution, Journal of the Minerals, Metals and Materials Society, \textbf{59} (2007), 83--90.

\bibitem{SHP06} P. Stefanovic, M. Haataja and N. Provatas, Phase-field crystals with elastic interactions.
Phys. Rev. Lett., \textbf{96} (2006), 225504.

\bibitem{SHP09} P. Stefanovic, M. Haataja and N. Provatas, Phase-field crystal study of deformation
and plasticity in nanocrystalline materials, Phys. Rev. E,\textbf{ 80} (2009), 046107.

\bibitem{WW10} C. Wang and S. Wise, Global smooth solutions of the three dimensional modified phase field crystal equation, Methods Appl. Anal., \textbf{17} (2010), 191--212.

\bibitem{WW11} C. Wang and S. Wise, An energy stable and convergent finite difference scheme for the modified phase-field crystal equation, SIAM J. Numer. Anal., \textbf{49} (2011), 945--969.

\bibitem{WWL09} S. Wise, C. Wang and J. Lowengrub, An energy-stable and convergent finite-difference scheme for the phase-field crystal equation, SIAM J. Numer. Anal., \textbf{47} (2009), 2269--2288.


\end{thebibliography}
\end{document}